\documentclass[12pt]{article}
\usepackage{amsthm}
\usepackage{amsmath}
\usepackage{enumerate}
\usepackage{amssymb}
\usepackage{latexsym}
\usepackage{amsfonts}
\usepackage{color}
\usepackage{mathrsfs}
\usepackage{epsfig}
\newtheorem{theorem}{\bf Theorem}[section]
\newtheorem{proposition}[theorem]{\bf Proposition}
\newtheorem{definition}[theorem]{\bf Definition}
\newtheorem{corollary}[theorem]{\bf Corollary}

\newtheorem{remark}[theorem]{\bf Remark}
\newtheorem{lemma}[theorem]{\bf Lemma}

\newsavebox{\savepar}

		\date{}
\begin{document}
		\title{Infinitely many solutions for a doubly-nonlocal fractional problem involving two critical nonlinearities}
	\author{ Akasmika Panda \& Debajyoti Choudhuri \\
		\small{Department of Mathematics, National Institute of Technology Rourkela,}\\ \small{Rourkela -769008, India}\\
		\small{Emails: dc.iit12@gmail.com, akasmika44@gmail.com}
	}
	\date{}
	\maketitle
	\begin{abstract}
	\noindent In this article, we study the existence of infinitely many nontrivial solutions for the following problem involving fractional $(p(x),p^+)$-Laplacian.
	\begin{eqnarray}
	\begin{split}
	(-\Delta)_{p(x)}^{s}u+(-\Delta)_{p^+}^{s}u&= |u|^{r(x)-2}u+|u|^{p_s^*(x)-2}u+\lambda f(x,u)~\text{in}~\Omega,\\
	u&= 0~\text{in}~\mathbb{R}^N\setminus\Omega.\nonumber
	\end{split}
	\end{eqnarray}
	Here  $\Omega\subset \mathbb{R}^N$ is a bounded domain, $\lambda>0$, $s\in (0,1),~p(\cdot,\cdot)$ is a continuous, bounded, symmetric function in $\mathbb{R}^N\times\mathbb{R}^N$ such that $ p^+=\underset{(x,y)\in\mathbb{R}^N\times\mathbb{R}^N}{\sup}p(x,y)<\frac{N}{s}$, $r\in C(\overline{\Omega})$ with $p^+<r^-\leq r^+\leq (p^+)^*_s=\frac{Np^+}{N-sp^+}$, $p_s^*(x)=\frac{Np(x,x)}{N-sp(x,x)}$ for every $x\in \mathbb{R}^N$ and the function $f$ satisfies certain assumptions which will be made precise later. Further, the exponents $p_s^*(\cdot)$ and $r(\cdot)$ are two critical exponents with the assumption that the critical sets $\{x\in\Omega:p_s^*(x)=(p^+)_s^*\}$ and $\{x\in\Omega:r(x)=(p^+)_s^*\}$ are nonempty. We also develop a concentration compactness type principle in the process. 
		{\bf keywords}:~Concentration compactness principle, fractional Sobolev space with variable exponent, fractional $p(x)$-Laplace operator, critical exponent, genus,
		symmetric mountain pass lemma.
	\end{abstract}
\tableofcontents
\section{Introduction}
The present paper is devoted to the following fractional elliptic PDE involving two critical exponents. 
\begin{eqnarray}\label{app}
\begin{split}
(-\Delta)_{p(x)}^{s}u+(-\Delta)_{p^+}^{s}u&= |u|^{r(x)-2}u+|u|^{p_s^*(x)-2}u+\lambda f(x,u)~\text{in}~\Omega,\\
u&= 0~\text{in}~\mathbb{R}^N\setminus\Omega,
\end{split}
\end{eqnarray}
where $\Omega\subset \mathbb{R}^N$ is a bounded domain, $s\in(0,1)$, $\lambda>0$ and $f:\overline{\Omega}\times\mathbb{R}\rightarrow\mathbb{R}$ satisfies a Carath\'{e}odory condition. The functions $p(\cdot,\cdot)$ and $r(\cdot)$ satisfy the following assumptions.
\begin{enumerate}
	\item[($\mathcal{P}_1$)] $p\in C(\mathbb{R}^N\times\mathbb{R}^N)$ is a symmetric function. 
	\item[($\mathcal{P}_2$)] $1<p^-=\underset{(x,y)\in\mathbb{R}^N\times\mathbb{R}^N}{\inf}p(x,y)\leq p(x,y)\leq p^+=\underset{(x,y)\in\mathbb{R}^N\times\mathbb{R}^N}{\sup}p(x,y)<\frac{N}{s}$ and  $p^+<(p^*_s)^-=\underset{x\in\overline{\Omega}}{\inf}~p_s^*(x)\leq p_s^*(x)=\frac{Np(x,x)}{N-sp(x,x)}$.
	\item[($\mathcal{P}_3$)]  $A_1=\{x\in\Omega:p_s^*(x)=(p^+)_s^*\}\neq\emptyset$ where $(p^+)^*_s=\frac{Np^+}{N-sp^+}$.
	\item[($\mathcal{R}_1$)] $r\in C(\overline{\Omega})$ with $p^+<r^-\leq r^+\leq (p^+)^*_s$.
	\item[($\mathcal{R}_2$)] $A_2=\{x\in\Omega:r(x)=(p^+)_s^*\}\neq \emptyset$.	
\end{enumerate} 
The fractional $p(x)$-Laplacian ($(-\Delta)_{p(x)}^{s}$) and the fractional $p^+$-Laplacian ($(-\Delta)_{p^+}^{s}$) are defined as
$$(-\Delta)_{p(x)}^{s}u=P.V.\int_{\mathbb{R}^N}\frac{|u(x)-u(y)|^{p(x,y)-2}(u(x)-u(y))}{|x-y|^{N+sp(x,y)}}dy$$ and $$(-\Delta)_{p^+}^{s}u=P.V.\int_{\mathbb{R}^N}\frac{|u(x)-u(y)|^{p^+-2}(u(x)-u(y))}{|x-y|^{N+sp^+}}dy$$ respectively. We name $\mathcal{L}=(-\Delta)_{p(x)}^{s}+(-\Delta)_{p^+}^{s}$ as the fractional  $(p(x),p^+)$-Laplacian. \\
The first objective of this paper is to derive a concentration compactness type principle (CCTP) for fractional Sobolev spaces in two critical exponent set up as stated in Theorem $\ref{conc}$ in Section $\ref{continuous}$. The second is to prove the existence of infinitely many small solutions to $\eqref{app}$ (see Theorem $\ref{exist}$) under suitable assumptions on $f$, which are given in Section $\ref{important}$. The strategy of the proof is based on the application of CCTP (Theorem $\ref{conc}$) and the symmetric mountain pass lemma (Lemma $\ref{symmetric}$). Here we choose $\lambda>0$, since for the case $\lambda\leq0$ and $p(\cdot,\cdot)$=constant, the validity of a Poho\v{z}aev identity rules out the existence of a nontrivial solution to $\eqref{app}$ in a star shaped domain. The present work is new in the sense that - to our knowledge - there is no existence result for the problem $\eqref{app}$, even for the local case, i.e. for $s=1$, where the approach can be closely adapted.\\
One of the most important theoretical developments in the theory of elliptic PDEs is due to the work of P. L. Lions (\cite{P.L} in 1984, \cite{Lion} and \cite{Lion1} in 1985). In his work he introduced the notion of concentration compactness principle (CCP) which became a fundamental method to show the existence of solutions of variational problems involving critical Sobolev exponents. A strong reason for the popularity of this principle is because it could address a way to compensate for the lack of compact embeddings amongst certain function spaces, which mostly resulted due to the presence of a critical exponent or due to the consideration of an unbounded domain. It aided to examine the nature of weakly convergent subsequence and determine the energy levels of variational problems below which the Palais-Smale condition is satisfied. Lions \cite{Lion} gave  a systematic theory to handle the issue of loss of compactness not only when it is lost due to translations but also because of the invariance of $\mathbb{R}^N$, for instance, by the non-compact group of dilations.\\
Later, in 1995, Chabrowski \cite{Chabrowski} extended the result of Lions for semilinear elliptic equations with critical and subcritical Sobolev exponent but at infinity. Palatucci \cite{Palatucci} developed a CCP which can be applicable to study a PDE involving a fractional Laplacian and a critical exponent term. Mosconi et al, further generalized the result due to \cite{Palatucci} which can be used to analyse equations involving fractional $p$-Laplacian with a critical growth \cite{Mosconi Perera} and with a nearly critical growth \cite{Mosconi}. At this point, we also refer the reader to the noteworthy work on CCP due to Dipierro et al \cite{Valdi1} (Proposition 3.2.3). A CCP was recently proposed by Bonder et al. \cite{Bonder1}, which can be used to study problems involving the fractional $p$-Laplacian operator for  $1<p<\frac{N}{s}$ in unbounded domain. An advanced version of CCP of P. L. Lions is obtained by Fu \cite{Fu} for variable exponent case dealing with  Dirichlet problems involving $p(x)$-Laplacian with critical exponent $p^*(x)=\frac{Np(x)}{N-p(x)}$. Moreover, Bonder and Silva \cite{Bonder2} developed a more general result for the variable exponent case where the exponent does not require to be critical everywhere. The author worked with the exponent $q(x)$ considering the set $\{x\in\Omega:q(x)=p^*(x)\}$ to be nonempty.\\
	In the recent years, an increased interest among the researchers has been observed to the study of the following type of elliptic equations
		\begin{eqnarray}\label{1}
	\begin{split}
	(-\Delta)_{p(x)}^{s}u&= g(x,u)~\text{in}~\Omega,\\
	u&= 0~\text{in}~\mathbb{R}^N\setminus\Omega,
	\end{split}
	\end{eqnarray} 
	where $(-\Delta)_{p(x)}^{s}$ is the fractional $p(x)$-Laplacian, $\Omega$ is bounded domain in $\mathbb{R}^N$, $p(\cdot,\cdot)$ is a bounded, continuous symmetric real valued function over $\mathbb{R}^N\times\mathbb{R}^N$ and the function $g$ has a subcritical growth. The solution space for the problem in $\eqref{1}$ is the fractional Sobolev space with variable exponent which is defined in Section $\ref{important}$. Readers may refer \cite{Bahrouni1},  \cite{Bahrouni2}, \cite{Ho}, \cite{Kaufmann} and the references therein for further readings on problems of the type as in $\eqref{1}$. Due to the non availability of embeddings from the fractional Sobolev space with variable exponent $p(\cdot,\cdot)$ to the space $L^{p_s^*(\cdot)}(\Omega)$ (where $p_s^*(\cdot)$ is the fractional critical exponent), it is difficult to prove the concentration compactness principle that deals with problem of type $\eqref{1}$ with some critical growth conditions on $g$. Hence, problem involving a fractional $p(x)$-Laplacian and having a critical nonlinearity is still an open problem.\\
	 When $p(\cdot,\cdot)$ is a constant function, i.e. $p(x,y)=q$ for every $(x,y)\in\mathbb{R}^N\times\mathbb{R}^N$ and $r(x)= p_s^*(x)= q_s^*$ for every $x\in \Omega$, the problem $\eqref{app}$ boils down to the following well known fractional problem.	
	\begin{eqnarray}\label{p}
\begin{split}
(-\Delta)_{q}^{s}u&= |u|^{q_s^*-2}u+\tilde{\lambda}g(x,u)~\text{in}~\Omega,\\
u&= 0~\text{in}~\mathbb{R}^N\setminus\Omega,
\end{split}
\end{eqnarray}
where $1<q<\frac{N}{s}$ and $g$ satisfies some subcritical growth conditions. The existence and multiplicity results for the problem $\eqref{p}$ for a range of $\lambda $ has been studied in \cite{Barrios,Cantizano,Fiscella,Mosconi,Perera,Rodrigo} and the bibliography therein. A multiplicity result for a Schr\"{o}dinger-Kirchhoff type problem with the fractional $p$-Laplacian and a critical exponent in $\mathbb{R}^N$ has been studied in \cite{Xiang}.\\
 In the literature, we find critical fractional $(p,q)$- Laplacian problems of the following type. \\ 
	\begin{eqnarray}\label{pq}
	\begin{split}
(-\Delta)_{p}^{s}u+(-\Delta)_{q}^{s}u&= |u|^{p_s^*-2}u+\tilde{\lambda}g(x,u)~\text{in}~\Omega,\\
	u&= 0~\text{in}~\mathbb{R}^N\setminus\Omega,
	\end{split}
	\end{eqnarray}
	with $1\leq q<p<\frac{N}{s}$. Ambrosio \& Isernia in \cite{Ambrosio} considered the problem $\eqref{pq}$ and used concentration compactness principle, symmetric mountain pass theorem to guarantee the existence of infinitely many solutions for a suitable range of $\tilde{\lambda}$.  It is also noteworthy to refer to the problem addressed by the authors in \cite{Bhakta} where they have discussed the existence of multiple nontrivial solutions of $(p,q)$ fractional Laplacian equations involving concave-critical type nonlinearities using CCP. The associated energy functional of the problem $\eqref{pq}$ involving fractional $(p,q)$-Laplcian is a double phase variational integral. One can refer \cite{Bahrouni3,Ragusa}  and the bibliography therein for further study of double phase functionals with variable exponents.
\section{Preliminaries and main result}\label{important}
Let $\Omega$ be a bounded domain in $\mathbb{R}^N$ and denote $$C_+(\overline{\Omega}\times\overline{\Omega})=\{f\in C(\overline{\Omega}\times\overline{\Omega}): ~1<f^-\leq f(x,y)\leq f^+<\infty,~\forall~(x,y)\in \overline{\Omega}\times\overline{\Omega}\}$$ where $f^+=\underset{\overline{\Omega}\times\overline{\Omega}}{\sup}f(x,y)~( \text{or}~\underset{\overline{\Omega}}{\sup}f(x))$, $f^-=\underset{\overline{\Omega}\times\overline{\Omega}}{\inf}f(x,y)~(\text{or}~\underset{\overline{\Omega}}{\inf}f(x))$.\\
Let $p\in C_+(\overline{\Omega})$ and $\nu$ be a complete, $\sigma$-finite measure in $\Omega$. The Lebesgue space with variable exponent is defined as
	$$L_\nu^{p(\cdot)}(\Omega)=\{u:\Omega\rightarrow\mathbb{R}~\text{is}~\nu~\text{measurable}~ :\int_{\Omega}|u|^{p(x)}d\nu<\infty\}$$
which is a Banach space endowed with the norm (see \cite{Fan})
	$$\|u\|_{L_\nu^{p(\cdot)}(\Omega)}=\inf\{\eta\in \mathbb{R}^+:\int_{\Omega}\Big|\frac{u(x)}{\eta}\Big|^{p(x)}d\nu<1\}.$$
	For $d\nu=dx$ we will denote the Lebesgue space with variable exponent as $L^{p(\cdot)}(\Omega)$ whose norm will be denoted by  $\|u\|_{L^{p(\cdot)}(\Omega)}$.\\
	We now give a few more notations and state some propositions which will be referred to henceforth very often.
\begin{proposition}[\cite{Ho}, Proposition 2.1]\label{holder}
	Let $f,g\in C_+(\overline{\Omega})$ with $f(x)\leq g(x)$ for every $x\in \overline{\Omega}$. Then, $$\|u\|_{L_\nu^{f(\cdot)}(\Omega)}\leq 2[1+\nu(\Omega)]\|u\|_{L_\nu^{g(\cdot)}(\Omega)},~\forall u\in L_\nu^{f(\cdot)}(\Omega)\cap L_\nu^{g(\cdot)}(\Omega).$$
\end{proposition}
	\begin{proposition}[$\cite{Fan}$]\label{holder2}
		1. (H\"{o}lder's Inequality) Let $\alpha,\theta,\gamma:\overline{\Omega}\rightarrow[1,\infty]$ with $\frac{1}{\alpha(x)}=\frac{1}{\theta(x)}+\frac{1}{\gamma(x)}$. If $h\in L^{\gamma(\cdot)}(\Omega)$ and $f\in L^{\theta(\cdot)}(\Omega)$, then 
			$$\|hf\|_{L^{\alpha(\cdot)}(\Omega)}\leq C \|h\|_{L^{\gamma(\cdot)}(\Omega)}\|f\|_{L^{\theta(\cdot)}(\Omega)}.$$
2. If $p,q\in C_+({\overline{\Omega}})$ and $p(x)\leq q(x)$, for $x\in{\overline{\Omega}}$, then $L^{q(\cdot)}(\Omega)\hookrightarrow L^{p(\cdot)}(\Omega)$ and this embedding is continuous. 
	\end{proposition}	
\noindent We fix the exponents $0<s<1$, $p\in C_+(\mathbb{R}^N\times\mathbb{R}^N)$, $\tilde{p}(x)=p(x,x)$ for every $x\in\mathbb{R}^N$ and assume that $p(\cdot,\cdot)$ is a symmetric function. We define the fractional Sobolev space with variable exponent and the corresponding  Gagliardo seminorm as follows (see \cite{Bahrouni1,Bahrouni2,Ho,Kaufmann}). 
	\begin{align}
	W^{s,p(\cdot,\cdot)}(\Omega)=\{u\in L^{\tilde{p}(\cdot)}(\Omega):\int_{\mathbb{R}^N}\int_{\mathbb{R}^N}\frac{|u(x)-u(y)|^{p(x,y)}}{\eta^{p(x,y)}|x-y|^{N+sp(x,y)}}dydx<\infty, \text{ for some }\eta>0\}.\nonumber
	\end{align}
	and
	$$[u]^{s,p(\cdot,\cdot)}=\inf\{\eta\in \mathbb{R}^+:\int_{\mathbb{R}^N}\int_{\mathbb{R}^N}\frac{|u(x)-u(y)|^{p(x,y)}}{\eta^{p(x,y)}|x-y|^{N+sp(x,y)}}dydx<1\}.$$
	The space $W^{s,p(\cdot,\cdot)}(\Omega)$ is a reflexive Banach space equipped with the norm 
	$$\|u\|_{W^{s,p(\cdot,\cdot)}(\Omega)}=\|u\|_{L^{\tilde{p}(\cdot)}(\Omega)}+[u]^{s,p(\cdot,\cdot)}.$$
	We define a subspace of $W^{s,p(\cdot,\cdot)}(\Omega)$, denoted as $W_0$, as follows
	$$W_0=\{u\in W^{s,p(\cdot,\cdot)}(\Omega):u=0 ~\text{a.e. in}~\mathbb{R}^N\setminus\Omega\}.$$ 
The norm on $W_0$ is defined as
	$$\|u\|_{W_0}=\inf\{\eta\in\mathbb{R}^+:\int_{\mathbb{R}^N}\int_{\mathbb{R}^N}\frac{|u(x)-u(y)|^{p(x,y)}}{\eta^{p(x,y)}|x-y|^{N+sp(x,y)}}dydx<1\}.$$
	\begin{lemma}[$\cite{Fan}$]\label{rho}
		Let $u,u_k\in W_0$, $k\in \mathbb{N}$ and define the modular function as
		$$\rho_{W_0}(u)=\int_{\mathbb{R}^N}\int_{\mathbb{R}^N}\frac{|u(x)-u(y)|^{p(x,y)}}{|x-y|^{N+sp(x,y)}}dydx.$$
		Then, we have the following relation between the modular function and the norm.
		\begin{enumerate}
			\item $\|u\|_{W_0}=\eta\iff \rho_{W_0}(\frac{u}{\eta})=1.$
			\item $\|u\|_{W_0}>1(<1,=1)\iff \rho_{W_0}(u)>1(<1,=1).$
			\item $\|u\|_{W_0}>1\implies\|u\|_{W_0}^{p^-}\leq\rho_{W_0}(u)\leq\|u\|_{W_0}^{p^+}.$
			\item $\|u\|_{W_0}<1\implies\|u\|_{W_0}^{p^+}\leq\rho_{W_0}(u)\leq\|u\|_{W_0}^{p^-}.$
			\item $\underset{{k\rightarrow\infty}}{\lim}\|u_k-u\|_{W_0}=0\iff\underset{{k\rightarrow\infty}}{\lim}\rho_{W_0}(u_k-u)=0.$
		\end{enumerate}
	\end{lemma}
\begin{remark}
	The notion of fractional Sobolev space with variable exponent is a generalization of fractional Sobolev space with constant exponent. Let $q\in(1,\infty)$, then we define the fractional Sobolev space with constant exponent as follows. $$W_0^{s,q}(\Omega)=\{u\in L^q(\mathbb{R}^N): \int_{\mathbb{R}^{2N}}\frac{|u(x)-u(y)|^{q}}{|x-y|^{N+sq}}dydx<\infty,~u=0 \text{ in }\mathbb{R}^N\setminus\Omega\}$$ endowed with the norm $$\|u\|_{s,q}^q=\int_{\mathbb{R}^N}\int_{\mathbb{R}^N}\frac{|u(x)-u(y)|^{q}}{|x-y|^{N+sq}}dydx.$$
	\end{remark}
\noindent Given below are a few well known propositions and theorems in the literature.
	\begin{proposition}[$\cite{Ho}$]\label{reflexive}
		The spaces $(W_0,\|.\|_{W_0})$ and $(W_0^{s,q}(\Omega),\|.\|_{s,q})$ are reflexive, uniformly convex Banach spaces.	
	\end{proposition}
\begin{theorem}[\cite{Valdinoci}]\label{constant}
	Let $0<s<1$ and $q\in[1,\infty)$ with $sq<N$. Then, there exists a constant $C>0$ depending on $N,s,q$ such that for any $u\in W_0^{s,q}(\Omega)$ we have
	$$C\|u\|_{L^r(\mathbb{R}^N)}^q\leq \int_{\mathbb{R}^{N}}\int_{\mathbb{R}^{N}}\frac{|u(x)-u(y)|^q}{|x-y|^{N+sq}}dxdy,~\forall~r\in[q,q_s^*]$$
 where $q_s^*=\frac{Nq}{N-sq}$ is the fractional critical Sobolev exponent. Moreover, the space $ W_0^{s,q}(\Omega)$ is continuously embedded in $L^r(\Omega)$ for every $r\in[1,q_s^*]$ and compactly embedded in $L^r(\Omega)$ for every $r\in[1,q_s^*)$.
\end{theorem}
	\begin{theorem}[\cite{Ho}]\label{poin}
		Let us assume $0<s<1$ and $p\in C_+(\mathbb{R}^N\times\mathbb{R}^N)$ such that $sp^+<N$. Then, for any $\beta\in C_+(\overline{\Omega})$ with $\beta(\cdot)<p^*_s(\cdot)$, there exists $C=C(p,s,N,\Omega,\beta)>0$ such that for every $u\in W_0$,
		$$\|u\|_{L^{\beta(\cdot)}(\Omega)}\leq C\|u\|_{W_0}.$$
		Moreover, the embedding from $W_0$ to $L^{\beta(\cdot)}(\Omega)$ is continuous and also compact.
	\end{theorem}
	\begin{remark}
		If $p>q$, then there is no continuous embedding result from $W_0^{s,p}(\Omega)$ to $W_0^{s,q}(\Omega)$, refer \cite{Mironescu}  for counterexamples. So, the space $W_0^{s,p^+}(\Omega)$ need not be embedded in $W_0$.
	\end{remark}
	\noindent We denote $X=W_0\cap W_0^{s,p^+}(\Omega)$ endowed with the norm $$\|u\|_X=\|u\|_{s,p^+}+\|u\|_{W_0}.$$
	\begin{lemma}\label{cap}
		1. The space $(X,\|.\|_X)$ is a reflexive Banach space.\\
		2. The embedding $X\hookrightarrow L^{r(\cdot)}(\Omega)$ for $r\in C_+(\overline{\Omega})$ is continuous when $r^+\leq (p^+)_s^*$ and is compact whenever $r^+<(p^+)_s^*$.
	\end{lemma}
	\noindent The proof of the above lemma is a straight forward application of Proposition $\ref{holder2}$, Proposition $\ref{reflexive}$ and Theorem $\ref{constant}$.\\
	\noindent We now define the notion of weak solution to the problem $\eqref{app}$.
		\begin{definition}
			A function $u\in X$ is said to be a weak solution of problem $\eqref{app}$, if $u$ is a critical point of the corresponding energy functional
			\begin{align}\label{func}
			\Psi(u)&=\int_{\mathbb{R}^{2N}}\frac{1}{p(x,y)}\frac{|u(x)-u(y)|^{p(x,y)}}{|x-y|^{N+sp(x,y)}}dydx+\int_{\mathbb{R}^{2N}}\frac{1}{p^+}\frac{|u(x)-u(y)|^{p^+}}{|x-y|^{N+sp^+}}dydx\nonumber\\&~~-\int_{\Omega}\frac{1}{r(x)}|u|^{r(x)}dx-\int_{\Omega}\frac{1}{p_s^*(x)}|u|^{p_s^*(x)}dx-\lambda\int_{\Omega}F(x,u),
			\end{align}
			where $F(x,t)=\int_{0}^{t}f(,\tau)d\tau$.
		\end{definition}
\noindent It is easy to show that $\Psi
\in C^1(X)$ and it's differential is given by
\begin{align}
\langle\Psi^\prime(u), v\rangle&=\int_{\mathbb{R}^{2N}}\frac{|u(x)-u(y)|^{p(x,y)-2}(u(x)-u(y))(v(x)-v(y))}{|x-y|^{N+sp(x,y)}}dydx\nonumber\\&~~+\int_{\mathbb{R}^{2N}}\frac{|u(x)-u(y)|^{p^+-2}(u(x)-u(y))(v(x)-v(y))}{|x-y|^{N+sp^+}}dydx\nonumber\\&~~-\int_{\Omega}|u|^{r(x)-2}uvdx-\int_{\Omega}|u|^{p_s^*(x)-2}uvdx-\lambda\int_{\Omega}f(x,u)vdx\nonumber,~\text{for any }~u,v\in X.
\end{align}
Denote $m=\min\{r^-,(p_s^*)^-\}$, refer $(\mathcal{P}_2)$ and $(\mathcal{R}_1)$. Let the function $f$ satisfies the following growth assumptions.
\begin{itemize}
	\item[($\mathcal{F}_1$)] $f:\Omega\times\mathbb{R}\rightarrow\mathbb{R}$ is a Carath\'{e}odory function, $f(x,0)=0~\text{a.e. in}~\Omega$, $f(\cdot,t)>0$ for all $t\in\mathbb{R}^+$ and $f(\cdot,-t)=-f(\cdot,t)$ for all $t\in \mathbb{R}$.
	\item[($\mathcal{F}_2$)] $|f(x,t)|\leq C(1+|t|^{\beta(x)-1}),~\text{a.a}~x\in\Omega$ and for all $ t\in\mathbb{R}$ where $\beta\in C(\overline{\Omega})$ and $p^+<\beta^-\leq \beta^+<m$.
	\item[($\mathcal{F}_3$)] $\lim\limits_{|t|\rightarrow 0^+}\frac{f(x,t)}{|t|^{p^--1}}=\infty$.
	\item[($\mathcal{F}_4$)] there exists $a>0$ and $\alpha\in C(\overline{\Omega})$ with $1<\alpha^-\leq \alpha^+<p^+$ such that
	$$mF(x,t)-f(x,t)t\leq a|t|^{\alpha(x)},~\text{for all}~x\in \Omega~\text{and}~t\in \mathbb{R}.$$
\end{itemize}
\noindent To prove the multiplicity result for $\eqref{app}$, we remind the symmetric mountain pass lemma followed by the definition of genus of a set.
\begin{definition}[Genus, \cite{Rabinowitz}]
Let $Z$ be a Banach space and $Y\subset Z$. We say $Y$ to be a symmetric set if $u\in Y$ implies that $-u\in Y$. For a closed and symmetric set $Y$ such that $0\notin Y$, we define the genus $\gamma(Y)$ of $Y$ by the smallest integer $k$
such that there exists an odd continuous mapping from $Y$ to $\mathbb{R}^k\setminus\{0\}$. If there does not exist such a $k$, we define $\gamma(Y)=\infty$. Further, $\gamma(\emptyset)=0$.	
\end{definition}
\noindent We set
$$\Gamma_k=\{Y\subset Z: Y\text{ is closed and symmetric},~0\notin Y~\text{and}~\gamma(Y)\geq k\}.$$
We have the following properties of genus from \cite{Rabinowitz}.
\begin{proposition}\label{properties}
	Let $Y$ and $\overline{Y}$ are two closed symmetric subsets of $Z$ such that $0\notin Y,\overline{Y}$. Then
	\begin{enumerate}
		\item If $Y\subset \overline{Y}$, then $\gamma(Y)\leq \gamma(\overline{Y})$.
		\item If there is an odd homeomorphism from $Y$ onto $\overline{Y}$, then $\gamma(Y)= \gamma(\overline{Y})$.
		\item If $\gamma(Y)<\infty$, then $\gamma(Y\setminus\overline{Y})\geq \gamma(Y)-\gamma(\overline{Y})$.
		\item $\gamma(S^{N-1})=N$, where $S^{N-1}$ is the sphere in $\mathbb{R}^N$.
		\item If $Y$ is compact then $\gamma(Y)<\infty$ and there exists $\delta>0$ such that $\gamma(Y)=\gamma(N_\delta(Y))$ where $N_\delta(Y)=\{x\in Z:\|x-Y\|\leq \delta\}$ is a closed and symmetric neighbourhood of $Y$.
	\end{enumerate}
	\begin{lemma}[Symmetric mountain pass lemma, \cite{Kajikiya}]\label{symmetric}
		Let $Z$ be an infinite dimensional Banach space and  $J\in C(Z,\mathbb{R})$ be a functional satisfying the conditions below:
		\begin{enumerate}
			\item $I$ is even, bounded from below, $I(0)=0$ and $I$ satisfies the $(PS)$ condition below certain energy level.\item For each $n\in \mathbb{N}$, there exists $Y_n\subset \Gamma_n$ such that $\underset{u\in Y_n}{\sup}~I(u)<0$.
		\end{enumerate}
		Then, either (a) or (b) below holds.
		\begin{enumerate}
		\item[(a)] There exists a sequence $\{u_n\}$ such that $I^\prime(u_n)=0$, $I(u_n)<0$ and $u_n\rightarrow 0$ in $Z$.
		\item[(b)] There exist two sequences $\{u_n\}$ and $\{v_n\}$ such that $I^\prime(u_n)=0$, $I(u_n)=0$, $u_n\neq 0$, $\lim\limits_{n\rightarrow\infty}u_n=0$, $I^\prime(v_n)=0$, $I(v_n)<0$, $\lim\limits_{n\rightarrow\infty}I(v_n)=0$ and $\{v_n\}_{n\in \mathbb{N}}$ converges to a non-zero limit.
		\end{enumerate}
	\end{lemma}
\end{proposition}
\noindent We take the help of the concentration compactness type principle, derived in Section $\ref{continuous}$, and the symmetric mountain pass lemma (Lemma $\ref{symmetric}$) to prove the existence of weak solutions to $\eqref{app}$. The main result of this article is stated below.
\begin{theorem}\label{exist}
	Let $\Omega$ is a bounded domain in $\mathbb{R}^N$ and $s\in (0,1)$. If the conditions $(\mathcal{P}_1)-(\mathcal{P}_3)$, $(\mathcal{R}_1)-(\mathcal{R}_2)$ and $(\mathcal{F}_1)-(\mathcal{F}_4)$ are fulfilled, then  there exists $\Lambda>0$ such that for any $\lambda\in (0,\Lambda)$, the problem $\eqref{app}$ admits a sequence of nontrivial weak solutions $\{u_n\}\subset X$ such that $u_n\rightarrow 0$ as $n\rightarrow \infty$. 	
\end{theorem}
\section{Concentration compactness type principle}\label{continuous}
 Following the original method discovered by P. L. Lions \cite{Lion} we derive a concentration compactness type principle which is given in Theorem $\ref{conc}$ below.
\begin{theorem}\label{conc}
	Let $\Omega$ be a bounded domain in $\mathbb{R}^N$,  $s\in (0,1)$ and $q\in(1,\infty)$ with $sq<N$. Let us consider two functions $r_1(\cdot)$, $r_2(\cdot)\in C(\overline{\Omega})$ such that
	\begin{equation}\label{cond1}
	q<r_1^-\leq r_1(x)\leq r_1^+\leq q^*=\frac{Nq}{N-sq}<\infty,
	\end{equation} 
	\begin{equation}\label{cond2}
	q<r_2^-\leq r_2(x)\leq r_2^+\leq q^*<\infty,
	\end{equation}
	and the {\it critical} sets $A_{r_1}=\{x\in\Omega:r_1(x)=q^*\}$, $A_{r_2}=\{x\in\Omega:r_2(x)=q^*\}$ are non empty.\\ 
	Assume $\{u_n\}$ to be a bounded sequence in $W_0^{s,q}(\Omega)$. Then, there exist $u\in W_0^{s,q}(\Omega)$ and bounded regular measures $\mu$, $\nu_1$, $\nu_2$ such that, up to a subsequence,
	\begin{center}
		$u_n\rightarrow u$ weakly in $W_0^{s,q}(\Omega)$ and strongly in $L^{\beta(\cdot)}(\Omega)$ for every $\beta\in C_+(\overline{\Omega})$ with $\beta^+<q^*,$
	\end{center}
	\begin{equation}\label{cc1}
	\int_{\mathbb{R}^N}\frac{|u_n(x)-u_n(y)|^{q}}{|x-y|^{N+sq}}dy\overset{t}{\rightharpoonup}\mu, ~|u_n|^{r_1(\cdot)}\overset{t}{\rightharpoonup}\nu_1,~|u_n|^{r_2(\cdot)}\overset{t}{\rightharpoonup}\nu_2
	\end{equation}
	where $\overset{t}{\rightharpoonup}$ denotes the tight convergence. Define a measure $\nu$ as $\nu=\nu_1+\nu_2$. Then, for some countable set $I$ we have
	\begin{equation}\label{cc2}
	\mu\geq \int_{\mathbb{R}^N}\frac{|u(x)-u(y)|^{q}}{|x-y|^{N+sq}}dy+\sum_{i\in I}\mu_i\delta_{x_i},~~\mu_i=\mu(\{x_i\}),
	\end{equation}
	\begin{equation}\label{cc3}
	\nu=|u|^{r_1(\cdot)}+|u|^{r_2(\cdot)}+\sum_{i\in I}\nu_i\delta_{x_i},~~\nu_i=\nu(\{x_i\}),
	\end{equation}			
	\begin{equation}\label{cc4}
	\nu_i\leq 2\max \left\{\frac{\mu_i^{q^*/q}}{S_{q,r_1,r_2}^{q^*}},\frac{\mu_i^{m/q}}{S_{q,r_1,r_2}^m}\right\}, ~\forall i\in I
	\end{equation}
	where $m=\min\{r_1^-,r_2^-\}$, $\{x_i:i\in I\}\subset A_{r_1}\cup A_{r_2}$,  $\{\nu_i:i\in I\}\subset(0,\infty)$ and $\{\mu_i:i\in I\}\subset (0,\infty)$. The constant $S_{q,r_1,r_2}=S_{q,r_1,r_2}(N, s, q,r_1,r_2,\Omega)>0$ is the Sobolev constant defined as
	\begin{equation}\label{best constant}
	S_{q,r_1,r_2}=\min\{S_{r_1},S_{r_2}\},
	\end{equation}
	where
	\begin{equation}\label{best constsnats}
	S_{r_1}=\underset{u\in W_0^{s,q}(\Omega)\setminus\{0\}}{\inf}\frac{\|u\|_{s,q}}{\|u\|_{L^{r_1(\cdot)}(\Omega)}}~\text{ and }~	S_{r_2}=\underset{u\in W_0^{s,q}(\Omega)\setminus\{0\}}{\inf}\frac{\|u\|_{s,q}}{\|u\|_{L^{r_2(\cdot)}(\Omega)}}.
	\end{equation}
\end{theorem}
\noindent Before proving the theorem we discuss some important results and definitions (refer \cite{Mosconi}).
\begin{definition}
	A bounded sequence $\{u_n\}$ is said to be tight if for every $\epsilon>0$, there exists a compact subset $K$ of $\mathbb{R}^N$ such that 
	$$\underset{n\in \mathbb{N}}{\sup}\int_{K^c}|u_n|dx<\epsilon.$$
\end{definition}
\noindent\textbf{Prokhorov's theorem}: Every bounded sequence $\{u_n\}$ are relatively sequentially compact if and only if the sequence is tight.
\begin{definition}
	A sequence of integrable functions $\{u_n\}$ in $\mathbb{R}^N$ converges tightly to a Borel regular measure $\nu$ if 
	$$\int_{\mathbb{R}^N}\varphi u_ndx\rightarrow\int_{\mathbb{R}^N}\varphi d\nu$$
	for all $\varphi\in C_b(\mathbb{R}^N)$, the space of bounded continuous functions in $\mathbb{R}^N$. We will symbolize the tight convergence by $\overset{t}{\rightharpoonup}$.
\end{definition} 
\begin{proof}[Proof of Theorem $\ref{conc}$]
	Let $\{u_n\}$ be a bounded sequence in $W_0^{s,q}(\Omega)$. Since $u_n\equiv 0$ in $\mathbb{R}^N\setminus\Omega$, the sequences $\{|u_n|^{r_1(\cdot)}\}$ and $\{|u_n|^{r_2(\cdot)}\}$ are tight. By the Prokhorov's theorem, there exist two positive Borel measures $\nu_1$ and $\nu_2$ such that $|u_n|^{r_1(\cdot)}\overset{t}{\rightharpoonup}\nu_1$ and $|u_n|^{r_2(\cdot)}\overset{t}{\rightharpoonup}\nu_2$. Apparently, $supp(\nu_1), supp(\nu_2)\subset\overline{\Omega}$. Denote
$$|D^su_n(x)|^{q}=	\int_{\mathbb{R}^N}\frac{|u_n(x)-u_n(y)|^{q}}{|x-y|^{N+sq}}dy.$$
The tightness of the sequence $\{|D^su_n|^{q}\}$ guarantees the existence of a Borel measure $\mu$ such that $|D^su_n|^{q}\overset{t}{\rightharpoonup}\mu$ and $\eqref{cc1}$ is proved.\\
The functions $r_1(\cdot),r_2(\cdot)\in C_+(\overline{\Omega})$ satisfy $\eqref{cond1}$, $\eqref{cond2}$ and the critical sets $A_{r_1}=\{x\in\Omega:r_1(x)=q^*\}$, $A_{r_2}=\{x\in\Omega:r_2(x)=q^*\}$ are non empty. Then, by using the Proposition $\ref{holder}$ and Theorem $\ref{constant}$, for any $\phi\in C_c^\infty(\mathbb{R}^N)$ with $0\leq|\phi|\leq1$ we have the following Sobolev inequalities.
$$S_{r_1}\|u_n\phi\|_{L^{r_1(\cdot)}(\Omega)}\leq \|u_n\phi\|_{s,q}~\text{ and }~S_{r_2}\|u_n\phi\|_{L^{r_2(\cdot)}(\Omega)}\leq \|u_n\phi\|_{s,q},$$
where $S_{r_1}$, $S_{r_2}$ are defined in $\eqref{best constsnats}$. Let us denote $S_{q,r_1,r_2}=\min\{S_{r_1},S_{r_2}\}$ and $m=\min\{r_1^-,r_2^-\}$. Observe that
\begin{align}\label{observe}
\int_{\Omega}\left(|u_n|^{r_1(x)}+|u_n|^{r_2(x)}\right)|\phi|^{q^*} dx&\leq \int_{\mathbb{R}^N}|u_n\phi|^{r_1(x)}+\int_{\mathbb{R}^N}|u_n\phi|^{r_2(x)}\nonumber\\& \leq \max\left\{\|u_n\phi\|_{L^{r_1(\cdot)}(\Omega)}^{q^*},\|u_n\phi\|_{L^{r_1(\cdot)}(\Omega)}^{m}\right\}\nonumber\\&~~~~+\max\left\{\|u_n\phi\|_{L^{r_2(\cdot)}(\Omega)}^{q^*},\|u_n\phi\|_{L^{r_2(\cdot)}(\Omega)}^{m}\right\}\nonumber\\&\leq 2\max\left\{\frac{\|u_n\phi\|_{s,q}^{q^*}}{S_{q,r_1,r_2}^{q^*}},\frac{\|u_n\phi\|_{s,q}^{m}}{S_{q,r_1,r_2}^{m}}\right\}.
\end{align}
Now using the Minkowski's inequality we get
\begin{align}\label{minkowski}
\|u_n\phi\|_{s,q}&
\leq 
\left(\int_{\mathbb{R}^{2N}}\frac{|u_n(x)|^{q}|\phi(x)-\phi(y)|^{q}}{|x-y|^{N+sq}}dydx\right)^{\frac{1}{q}}\nonumber\\&~~~~+\left(\int_{\mathbb{R}^{2N}}\frac{|\phi(y)|^{q}|u_n
	(x)-u_n(y)|^{q}}{|x-y|^{N+sq}}dydx\right)^{\frac{1}{q}}\nonumber\\&
\leq\left(\int_{\mathbb{R}^{2N}}|u_n(x)|^{q}|D^s\phi(x)|^{q}dx\right)^{\frac{1}{q}}+\left(\int_{\mathbb{R}^{2N}}|\phi(x)|^{q}|D^su_n(x)|^{q}dx\right)^{\frac{1}{q}}. 
\end{align}
Since $\{u_n\}$ is bounded in $W_0^{s,q}(\Omega)$, there exists $u\in W_0^{s,q}(\Omega)$ and a subsequence, still denoted as $\{u_n\}$, such that $u_n$ converges weakly to $u$ in $W_0^{s,q}(\Omega)$ and strongly in $L^{\beta(\cdot)}(\mathbb{R}^N)$ for any $\beta(\cdot)<q^*$. By the Corollary $\ref{Linfinity p^+}$, we observe $\int_{\mathbb{R}^{N}}\frac{|\phi(x)-\phi(y)|^{q}}{|x-y|^{N+sq}}dy\in L^\infty(\mathbb{R}^N)$. Then, letting $n\rightarrow \infty$ and using $\eqref{cc1}$ in $\eqref{minkowski}$, we have
\begin{equation}\label{right}
\lim_{n\rightarrow\infty}\sup	\|u_n\phi\|_{s,q}\leq\left(\int_{\mathbb{R}^{N}}|u(x)|^{q}|D^s\phi(x)|^{q}dx\right)^{\frac{1}{q}}+\left(\int_{\mathbb{R}^{N}}|\phi(x)|^{q}d\mu\right)^{\frac{1}{q}}=M.
\end{equation}
We denote the measure $\nu=\nu_1+\nu_2$ and hence $|u_n|^{r_1(\cdot)}+|u_n|^{r_2(\cdot)}\overset{t}{\rightharpoonup}\nu$. Thus, we have
\begin{equation}\label{left}
\lim_{n\rightarrow \infty}\sup\int_{\mathbb{R}^{N}}\left(|u_n|^{r_1(x)}+|u_n|^{r_2(x)}\right)|\phi|^{q^*}dx=\int_{\mathbb{R}^N}|\phi|^{q^*}d\nu.
\end{equation}
By considering the inequalities $\eqref{right}$, $\eqref{left}$ and by passing the limit $n\rightarrow\infty$ in $\eqref{observe}$, we establish the following.
\begin{align}\label{c5}
\|\phi\|_{L_\nu^{q^*}(\mathbb{R}^N)}^{q^*}\leq2\max\left\{\frac{M^{q^*}}{S_{q,r_1,r_2}^{q^*}},\frac{M^m}{S_{q,r_1,r_2}^m}\right\}.
\end{align}
Suppose that $u=0$ and $u_n\rightarrow 0$ weakly in $W_0^{s,q}(\Omega)$, then $\eqref{c5}$ becomes
\begin{align}\label{holder ineq}
\|\phi\|_{L_\nu^{q^*}(\mathbb{R}^N)}^{q^*}\leq2\max\left\{\frac{\|\phi\|_{L_\mu^{q}(\Omega)}^{q^*}}{S_{q,r_1,r_2}^{q^*}},\frac{\|\phi\|_{L_\mu^q(\Omega)}^m}{S_{q,r_1,r_2}^m}\right\}.
\end{align}
Thus, by following the proof of the Lemma $\ref{lion}$ (stated in the Appendix), we guarantee the existence of a set of distinct points $\{x_i:i\in I\}\subset\mathbb{R}^N$ and $\{\nu_i:i\in I\}\subset(0,\infty)$ such that $$\nu=\sum_{i\in I}\nu_i\delta_{x_i},~\nu_i=\nu(\{x_i\}).$$
Suppose that $u\neq0$. Then, the sequence $\{v_n\}$, where $v_n=u_n-u$, is bounded in $W_0^{s,q}(\Omega)$ and there exists a subsequence of $\{v_n\}$ (named as $\{v_n\}$) which converges weakly to 0 in $W_0^{s,q}(\Omega)$. By the Br\'{e}zis-Lieb lemma [Lemma $\ref{brezis}$ in the Appendix] we have the following
\begin{align}
&\lim_{n\rightarrow\infty}\left(\int_{\mathbb{R}^N}|\phi|\left(|u_n|^{r_1(x)}+|u_n|^{r_2(x)}\right)dx-\int_{\mathbb{R}^N}|\phi|\left(|v_n|^{r_1(x)}+|v_n|^{r_2(x)}\right)dx\right)\nonumber\\&=\int_{\mathbb{R}^N}|\phi|\left(|u|^{r_1(x)}+|u|^{r_2(x)}\right)dx
\end{align}
for every $\phi\in C_c^\infty(\mathbb{R}^N)$. Clearly the sequences $\{|v_n|^{r_1(\cdot)}\}$ and $\{|v_n|^{r_2(\cdot)}\}$ are tight. Hence, on similar lines the representation of $\nu$ is given as
\begin{equation}\label{nu}
\nu=|u|^{r_1(\cdot)}+|u|^{r(\cdot)}+\sum_{i\in I}\nu_i\delta_{x_i}.
\end{equation} 
This proves $\eqref{cc3}$. We now make the following claim.\\
{\it Claim:} The set $\{x_i:i\in I\}$ is a subset of the {\it critical} set $A=A_{r_1}\cup A_{r_2}$.\\
{\it Proof.} We prove this claim by contradiction. Suppose that $x_i$, for a fixed $i\in I$, does not belong to the set $A$. Consider a ball with centre $x_i$ and radius $r>0$ such that $B_i=B_r(x_i)\subset\subset\Omega\setminus A$. Thus, $r_1(x)<q^*$ and $r_2(x)<q^*$ for every $x\in \overline{B_i}$. Further, the sequence $\{u_n\}$ converges strongly to $u$ in $L^{r_1(\cdot)}(B_i)$ and in $L^{r_2(\cdot)}(B_i)$. This implies $|u_n|^{r_1(\cdot)}+|u_n|^{r_2(\cdot)}\rightarrow|u|^{r_1(\cdot)}+|u|^{r_2(\cdot)}$ strongly in $L^1(B_i)$ which is a contradiction to the assumption that $x_i\in B_i$ (see $\eqref{nu}$).\\
Let us consider $\phi\in C_c^\infty(\mathbb{R}^N)$ such that $0\leq\phi\leq1,~\phi(0)=1$ and $supp(\phi)\subset B_1(0)$. For a fixed $j$, choose $\epsilon>0$ such that  for $i\neq j$, $B_\epsilon(x_i)\cap B_\epsilon(x_j)=\emptyset$. We define $\phi_{\epsilon,j}(x)=\phi(\frac{x-x_j}{\epsilon})$. Then, according to Bonder et al. in \cite{Ambrosio1,Bonder2}, we obtain the following.
\begin{equation}\label{claim}
\underset{\epsilon\rightarrow0}{\lim}\int_{\mathbb{R}^{N}}|u(x)|^{q}|D^s\phi_{\epsilon,j}(x)|^{q}dx=0.
\end{equation} 
In order to prove $\eqref{cc4}$, we first observe that 
\begin{align}
\int_{\mathbb{R}^N}|\phi_{\epsilon,j}|^{q^*}d\nu=\int_{\mathbb{R}^N}|\phi_{\epsilon,j}|^{q^*}\left(|u|^{r_1(x)}+|u|^{r_2(x)}\right)dx+\sum_{i\in I}\nu_i|\phi_{\epsilon,j}(x_i)|^{q^*}
 \geq \nu_j.\nonumber
\end{align}
and from $\eqref{claim}$ we obtain $$M=\left(\int_{\mathbb{R}^{N}}|u(x)|^{q}|D^s\phi_{\epsilon,j}(x)|^{q}dx\right)^{\frac{1}{q}}+\left(\int_{\mathbb{R}^{N}}|\phi_{\epsilon,j}(x)|^{q}d\mu\right)^{\frac{1}{q}}
\leq\mu(B_\epsilon(x_j))^{\frac{1}{q}}\rightarrow\mu_{j}^{\frac{1}{q}}~~~\text{as}~\epsilon\rightarrow 0.$$
Thus, on passing the limit $\epsilon\rightarrow0$ in $\eqref{c5}$, we establish $\eqref{cc4}$, i.e. we have
$$\nu_j\leq2\max\left\{\frac{\mu_j^{\frac{q^*}{q}}}{S_{q,r_1,r_2}^{q^*}},\frac{\mu_j^{\frac{m}{q}}}{S_{q,r_1,r_2}^m}\right\}.$$
We are now left to prove $\eqref{cc2}$. We already have $\mu\geq\sum_{i\in I}\mu_i\delta_{x_i}$. On using the weak convergence and the Fatou's lemma we obtain $\mu\geq|D^su(x)|^{q}$. Since, $\sum_{i\in I}\mu_i\delta_{x_i}$ and $|D^su(x)|^{q}$ are orthogonal measures, we conclude $\mu\geq|D^su(x)|^{q}+\sum_{i\in I}\mu_i\delta_{x_i}$ and this completes the proof.
\end{proof}
\begin{remark}
By the same argument, we can prove Theorem $\ref{conc}$ with finite number of critical exponents $r_m(\cdot)$ for $1\leq m\leq M$, instead of two critical exponents $r_1(\cdot)$ and $r_2(\cdot)$.
\end{remark}
 \begin{remark}
 	If the critical sets $A_{r_1}=A_{r_2}=\Omega$, i.e. if $r_1(x)=r_2(x)=q^*$ for every $x\in \Omega$, then the CCTP derived in Theorem $\ref{conc}$ is a consequence of the CCP proved in \cite{Mosconi}. Moreover, our result in Theorem $\ref{conc}$ is slightly more general than the CCP in \cite{Mosconi}, since we do not need the functions $r_1(\cdot)$ and $r_2(\cdot)$ to be critical everywhere. We also establish that the delta functions are located in the critical set $A=A_{r_1}\cup A_{r_2}$. 
 \end{remark}
\section{Proof of main result - Theorem $\ref{exist}$}\label{appl}
 We use the symmetric mountain pass lemma and the concentration compactness type principle to prove Theorem $\ref{exist}$. Hence, we need to first show that the corresponding functional $\Psi$ satisfies the Palais-Smale (P-S) condition below a certain energy level.
\begin{lemma}\label{bound}
	Let $(\mathcal{P}_1)-(\mathcal{P}_3)$, $(\mathcal{R}_1)-(\mathcal{R}_2)$ and $(\mathcal{F}_1)-(\mathcal{F}_4)$ hold. Then. every Palais-Smale (P-S) sequence is bounded in $X$.
\end{lemma}
\begin{proof}
	Let $\{u_n\}\subset X$ be a (P-S) sequence of the functional $\Psi$, defined in $\eqref{func}$, i.e. $\Psi(u_n)\rightarrow c$ and $\Psi^\prime(u_n)\rightarrow 0$ as $n\rightarrow\infty$. Thus, from Lemma $\ref{cap}$ and $(\mathcal{F}_4)$ we have
	\begin{align}\label{bounded p-s}
	c+o(1)&\geq \Psi(u_n)-\frac{1}{m}\langle\Psi^\prime(u_n),u_n\rangle\nonumber\\&= \int_{\mathbb{R}^{2N}}\left(\frac{1}{p(x,y)}-\frac{1}{m}\right)\frac{|u_n(x)-u_n(y)|^{p(x,y)}}{|x-y|^{N+sp(x,y)}}dydx+\left(\frac{1}{p^+}-\frac{1}{m}\right)\int_{\mathbb{R}^{2N}}\frac{|u_n(x)-u_n(y)|^{p^+}}{|x-y|^{N+sp^+}}dydx\nonumber\\&~~~~+\int_{\Omega}\left(\frac{1}{m}-\frac{1}{r(x)}\right)|u_n|^{r(x)}dx+\int_{\Omega}\left(\frac{1}{m}-\frac{1}{p_s^*(x)}\right)|u_n|^{p_s^*(x)}dx\nonumber\\&~~~~-\lambda\int_{\Omega}\left(F(x,u_n)-\frac{1}{m}f(x,u_n)u_n\right)dx\nonumber\\&\geq \left(\frac{1}{p^+}-\frac{1}{m}\right)\left(\int_{\mathbb{R}^{2N}}\frac{|u_n(x)-u_n(y)|^{p(x,y)}}{|x-y|^{N+sp(x,y)}}dydx+\int_{\mathbb{R}^{2N}}\frac{|u_n(x)-u_n(y)|^{p^+}}{|x-y|^{N+sp^+}}dydx\right)\nonumber\\&~~~~-\lambda\frac{a}{m}\int_{\Omega}|u_n|^{\alpha(x)}dx\nonumber\\&\geq \left(\frac{1}{p^+}-\frac{1}{\beta^-}\right)\left(\min\{\|u_n\|_{W_0}^{p^+},\|u_n\|_{W_0}^{p^-}\}+\|u_n\|_{s,p^+}^{p^+}\right)-C\lambda\frac{a}{m}\|u_n\|_{X}^{\alpha^\pm}.
	\end{align}
Let us assume that $\|u_n\|_X\rightarrow\infty$ as $n\rightarrow\infty$. Then, we observe $\frac{1}{\|u_n\|_X^{p^+}}=o(1)$. From $\eqref{bounded p-s}$, we obtain
\begin{equation}\label{contradiction}
o(1)\geq \frac{1}{\|u_n\|_X^{p^+}}\left(\frac{1}{p^+}-\frac{1}{\beta^-}\right)\left(\min\{\|u_n\|_{W_0}^{p^+},\|u_n\|_{W_0}^{p^-}\}+\|u_n\|_{s,p^+}^{p^+}\right)-C\lambda\frac{a}{m}\|u_n\|_{X}^{\alpha^\pm-p^+}.
\end{equation}
Since $\alpha^-\leq \alpha^+<p^+$, we produce a contradiction from $\eqref{contradiction}$ and hence $\{u_n\}$ is a bounded sequence in $X$.
\end{proof}
\noindent Note that for any $p\in C_+(\mathbb{R}^N\times\mathbb{R}^N)$, it is easy to show that $$(p^-)_s^*\leq(p_s^*)^-=\inf_{x\in \mathbb{R}^N} p^*_s(x)\leq p_s^*(x)\leq\sup_{x\in \mathbb{R}^N} p^*_s(x)=(p_s^*)^+\leq(p^+)_s^*.$$
\begin{lemma}\label{palais}
Assume $(\mathcal{P}_1)-(\mathcal{P}_3)$, $(\mathcal{R}_1)-(\mathcal{R}_2)$ and $(\mathcal{F}_1)-(\mathcal{F}_4)$ are satisfied. Then, there exist $C_1,C_2>0$ such that the functional $\Psi$ satisfies the Palais-Smale (P-S) condition for the energy level $$c<\left(\frac{1}{p^+}-\frac{1}{m}\right)\min\left\{2^{\frac{sp^+-N}{sp^+}}S_{p^+,r,p_s^*}^{\frac{N}{s}},2^{\frac{p^+}{p^+-m}}S_{p^+,r,p_s^*}^{\frac{mp^+}{m-p^+}} \right\}-C_1\lambda^{C_2}.$$ Here $S_{p^+,r,p_s^*}$ is the Sobolev constant given in $\eqref{best constant}$ and $m=\min\{r^-,(p_s^*)^-\}$.
\end{lemma}
\begin{proof}
	Let $\{u_n\}$ be a (P-S) sequence in $X$. Then, by Lemma $\ref{bound}$, $\{u_n\}$ is bounded in $X$ and thus there exists $u\in X$ such that, up to a subsequence (still denoted as $\{u_n\}$), $\{u_n\}$ converges weakly to $u$ in $X$. We need to show that $u_n\rightarrow u$ strongly in $X$. Since $\{u_n\}$ is bounded in $X$, $\{u_n\}$ is also bounded in $W_0$ and in $W_0^{s,p^+}(\Omega)$. Thus, using the concentration compactness type principle, stated in Theorem $\ref{conc}$, we have 
	\begin{equation}
	\int_{\mathbb{R}^N}\frac{|u_n(x)-u_n(y)|^{p^+}}{|x-y|^{N+sp^+}}dy\overset{t}{\rightharpoonup}\mu\geq\int_{\mathbb{R}^N}\frac{|u(x)-u(y)|^{p^+}}{|x-y|^{N+sp^+}}dy+\sum_{i\in I}\mu_{j}\delta_{x_i}, \nonumber
	\end{equation}
	\begin{equation}
	|u_n|^{r(\cdot)}+|u_n|^{p_s^*(\cdot)}\overset{t}{\rightharpoonup}\nu=|u|^{(r(\cdot)}+|u|^{p_s^*(\cdot)}+\sum_{i\in I}\nu_i\delta_{x_i}\nonumber
	\end{equation}	and		
	\begin{equation}
\nu_j\leq2\max\left\{\frac{\mu_j^{\frac{(p^+)_s^*}{p^+}}}{S_{p^+,r,p_s^*}^{(p^+)_s^*}},\frac{\mu_j^{\frac{m}{p^+}}}{S_{p^+,r,p_s^*}^{m}}\right\}, ~\forall j\in I,\nonumber
	\end{equation}
	where $m=\min\{r^-,(p_s^*)^-\}$, $I$ is countable, $\{x_j\}_{j\in I}\subset A=A_{1}\cup A_{2}$, $\{\nu_j\}_{j\in I}, \{\mu_j\}_{j\in I}\subset(0,\infty)$ and the constant $S_{p^+,r,p_s^*}$ is defined in $\eqref{best constant}$. 
	Let us denote
	$$|U_n(x)|=	\int_{\mathbb{R}^N}\frac{|u_n(x)-u_n(y)|^{p(x,y)}}{|x-y|^{N+sp(x,y)}}dy$$
	and consider an open, bounded subset $D$ of $\mathbb{R}^N$, defined as in the proof of Theorem $\ref{conc}$. Then, for any $x\in D^c$ and $y\in \Omega$, there exists a constant $C_d$ such that $|x-y|\geq C_d|x|$ and
	\begin{align}
	|U_n(x)|&=	\int_{\Omega}\frac{|u_n(y)|^{p(x,y)}}{|x-y|^{N+sp(x,y)}}dy\nonumber\\
	& \leq 	\int_{\Omega}\frac{|u_n(y)|^{p(x,y)}}{(C_d|x|)^{N+sp(x,y)}}dy\nonumber\\ & \leq C\max\left(\frac{1}{|x|^{N+sp^+}}, \frac{1}{|x|^{N+sp^-}}\right).\nonumber
	\end{align}
	Therefore, the sequence $\{|U_n|\}$ is tight and there exists a positive bounded Borel measure $\sigma$ such that $|U_n|\overset{t}{\rightharpoonup}\sigma$. Consider $\phi\in C_c^\infty(\mathbb{R}^N)$ with $0\leq\phi\leq1$, $\phi(0)=1$ and support in the unit ball of $\mathbb{R}^N$. Define $\phi_{\epsilon,j}=\phi(\frac{x-x_j}{\epsilon})$. Since $\{u_n\}$ is a (P-S) sequence, we have $\Psi(u_n)\rightarrow c$ and $\Psi^\prime(u_n)\rightarrow 0$ as $n\rightarrow\infty$. On the other hand 
	\begin{align}\label{a1}
	\langle\Psi^\prime(u_n), \phi_{\epsilon,j}u_n\rangle&=\int_{\mathbb{R}^{2N}}\frac{|u_n(x)-u_n(y)|^{p(x,y)-2}(u_n(x)-u_n(y))((\phi_{\epsilon,j}u_n)(x)-(\phi_{\epsilon,j}u_n)(y))}{|x-y|^{N+sp(x,y)}}dydx\nonumber\\&~~+\int_{\mathbb{R}^{2N}}\frac{|u_n(x)-u_n(y)|^{p^+-2}(u_n(x)-u_n(y))((\phi_{\epsilon,j}u_n)(x)-(\phi_{\epsilon,j}u_n)(y))}{|x-y|^{N+sp^+}}dydx\nonumber\\&~~-\int_{\Omega}|u_n|^{r(x)-2}u_n(\phi_{\epsilon,j}u_n)dx-\int_{\Omega}|u_n|^{p_s^*(x)-2}u_n(\phi_{\epsilon,j}u_n)dx\nonumber\\&~~-\lambda\int_{\Omega}f(x,u_n)(\phi_{\epsilon,j}u_n)dx\nonumber\\&=\int_{\mathbb{R}^{2N}}\frac{|u_n(x)-u_n(y)|^{p(x,y)-2}(u_n(x)-u_n(y))u_n(y)(\phi_{\epsilon,j}(x)-\phi_{\epsilon,j}(y))}{|x-y|^{N+sp(x,y)}}dxdy\nonumber\\&~~+\int_{\mathbb{R}^{2N}}\frac{|u_n(x)-u_n(y)|^{p^+-2}(u_n(x)-u_n(y))u_n(y)(\phi_{\epsilon,j}(x)-\phi_{\epsilon,j}(y))}{|x-y|^{N+sp^+}}dxdy\nonumber\\&~~+\int_{\mathbb{R}^{2N}}\frac{|u_n(x)-u_n(y)|^{p(x,y)}\phi_{\epsilon,j}(x)}{|x-y|^{N+sp(x,y)}}dxdy+\int_{\mathbb{R}^{2N}}\frac{|u_n(x)-u_n(y)|^{p^+}\phi_{\epsilon,j}(x)}{|x-y|^{N+sp^+}}dxdy\nonumber\\&~~-\int_{\Omega}|u_n|^{r(x)}\phi_{\epsilon,j}dx-\int_{\Omega}|u_n|^{p_s^*(x)}\phi_{\epsilon,j}dx-\lambda\int_{\Omega}f(x,u_n)u_n\phi_{\epsilon,j}dx.
	\end{align}
	We denote $$H_n(x,y)=\frac{|u_n(x)-u_n(y)|}{|x-y|^{\frac{N+sp(x,y)}{p(x,y)}}}~\text{ and }~\Phi_n(x,y)=\frac{|u_n(y)||\phi_{\epsilon,j}(x)-\phi_{\epsilon,j}(y)|}{|x-y|^{\frac{N+sp(x,y)}{p(x,y)}}}.$$ Since, $\{u_n\}$ is a bounded sequence in $W_0$, on using the H\"{o}lder's inequality on the first term in the right hand side of $\eqref{a1}$ we observe
	\begin{align}\label{a}
	\Big|\int_{\mathbb{R}^{2N}}&\frac{|u_n(x)-u_n(y)|^{p(x,y)-2}(u_n(x)-u_n(y))u_n(y)(\phi_{\epsilon,j}(x)-\phi_{\epsilon,j}(y))}{|x-y|^{N+sp(x,y)}}dxdy\Big|\nonumber\\&\leq \int_{\mathbb{R}^{2N}}\frac{|u_n(x)-u_n(y)|^{p(x,y)-1}|u_n(y)||\phi_{\epsilon,j}(x)-\phi_{\epsilon,j}(y)|}{|x-y|^{N+sp(x,y)}}dxdy\nonumber\\&=\int_{\mathbb{R}^{2N}}\left(\frac{|u_n(x)-u_n(y)|}{|x-y|^{\frac{N+sp(x,y)}{p(x,y)}}}\right)^{p(x,y)-1}\left(\frac{|u_n(y)||\phi_{\epsilon,j}(x)-\phi_{\epsilon,j}(y)|}{|x-y|^{\frac{N+sp(x,y)}{p(x,y)}}}\right)dxdy\nonumber\\&\leq C \||H_n|^{p(\cdot,\cdot)-1}\|_{L^{\frac{p(\cdot,\cdot)}{p(\cdot,\cdot)-1}}(\mathbb{R}^N\times\mathbb{R}^N)}\|\Phi_n\|_
	{L^{p(\cdot,\cdot)}(\mathbb{R}^N\times\mathbb{R}^N)}\nonumber\\&\leq C \||H_n|^{p(\cdot,\cdot)-1}\|_{L^{\frac{p(\cdot,\cdot)}{p(\cdot,\cdot)-1}}(\mathbb{R}^N\times\mathbb{R}^N)}\left(\int_{\mathbb{R}^{2N}}\frac{|u_n(y)|^{p(x,y)}|\phi_{\epsilon,j}(x)-\phi_{\epsilon,j}(y)|^{p(x,y)}}{|x-y|^{N+sp(x,y)}}dxdy\right)^{\frac{1}{k}}\nonumber\\& \leq C\left(\int_{\mathbb{R}^{2N}}\frac{\left(|u_n(x)|^{p^+}+|u_n(x)|^{p^-}\right)|\phi_{\epsilon,j}(x)-\phi_{\epsilon,j}(y)|^{p(x,y)}}{|x-y|^{N+sp(x,y)}}dxdy\right)^{\frac{1}{k}}
	\end{align}
	where $k$ is either $p^+$ or $p^-$. Using the weak convergence $u_n\rightarrow u$ in $W_0^{s,p^+}(\Omega)$, we have $u_n\rightarrow u$ strongly in $L^{p^+}(\Omega)$ and in $L^{p^-}(\Omega)$. Further, according to Lemma $\ref{lemma1}$, $\int_{\mathbb{R}^{N}}\frac{|\phi_{\epsilon}(x)-\phi_{\epsilon}(y)|^{p(x,y)}}{|x-y|^{N+sp(x,y)}}dx\in L^{\infty}(\mathbb{R}^N)$. Thus, applying limit $n\rightarrow\infty$ in the above inequality $\eqref{a}$ we get 
	\begin{align}\label{aa}
	\lim_{n\rightarrow \infty}	\Big|\int_{\mathbb{R}^{2N}}&\frac{|u_n(x)-u_n(y)|^{p(x,y)-2}(u_n(x)-u_n(y))u_n(y)(\phi_{\epsilon,j}(x)-\phi_{\epsilon,j}(y))}{|x-y|^{N+sp(x,y)}}dxdy\Big|\nonumber\\&\leq C\left(\int_{\mathbb{R}^{2N}}\frac{\left(|u(x)|^{p^+}+|u(x)|^{p^-}\right)|\phi_{\epsilon,j}(x)-\phi_{\epsilon,j}(y)|^{p(x,y)}}{|x-y|^{N+sp(x,y)}}dxdy\right)^{\frac{1}{k}}.
	\end{align} 
 We now make the following claim.\\
	{\it Claim:} $$\lim_{\epsilon\rightarrow0}\int_{\mathbb{R}^{2N}}\frac{\left(|u(x)|^{p^+}+|u(x)|^{p^-}\right)|\phi_{\epsilon,j}(x)-\phi_{\epsilon,j}(y)|^{p(x,y)}}{|x-y|^{N+sp(x,y)}}dxdy=0.$$
		Without loss of generality, we assume that $x_j=0$ and denote $\phi_\epsilon=\phi_{\epsilon,j}$. Using Lemma $\ref{lemma1}$ we have 
		\begin{equation}
			\int_{\mathbb{R}^N}\frac{|\phi_{\epsilon}(x)-\phi_{\epsilon}(y)|^{p(x,y)}}{|x-y|^{N+sp(x,y)}}dy\leq C\min\left(1/\epsilon^{sp^+}+1/\epsilon^{sp^-}, (\epsilon^N+ \epsilon^{N+s(p^--p^+)})|x|^{-(N+sp^-)}\right).\nonumber
		\end{equation}
		Thus, 
		\begin{align}\label{2}
			&\int_{\mathbb{R}^{2N}}\left(|u(x)|^{p^+}+|u(x)|^{p^-}\right)\frac{|\phi_\epsilon(x)-\phi_\epsilon(y)|^{p(x,y)}}{|x-y|^{N+sp(x,y)}}dydx\nonumber\\& \leq C\left(\frac{1}{\epsilon^{sp^+}}+\frac{1}{\epsilon^{sp^-}}\right)\int_{|x|<2\epsilon}|u(x)|^{p^+}+|u(x)|^{p^-}dx\nonumber	\end{align}\begin{align}&~~~~+ C(\epsilon^N+ \epsilon^{N+s(p^--p^+)})\int_{|x|\geq2\epsilon}\frac{|u(x)|^{p^+}+|u(x)|^{p^-}}{|x|^{N+sp^-}}dx\nonumber\\& =C(I+II).
		\end{align}
		We observe 
	\begin{align}
		I&=\left(\frac{1}{\epsilon^{sp^+}}+\frac{1}{\epsilon^{sp^-}}\right)\int_{|x|<2\epsilon}|u(x)|^{p^+}+|u(x)|^{p^-}dx \nonumber\\&\leq\left(\frac{1}{\epsilon^{sp^+}}+\frac{1}{\epsilon^{sp^-}}\right)\left(\||u|^{p^+}\|_{L^{\frac{(p^+)_s^*}{p^+}}(B_{2\epsilon}(0))}+\||u|^{p^-}\|_{L^{\frac{(p^+)_s^*}{p^+}}(B_{2\epsilon}(0))}\right)\|1\|_{L^{\frac{N}{sp^+}}(B_{2\epsilon}(0))}\nonumber\\&\leq C\epsilon^{sp^+}\left(\frac{1}{\epsilon^{sp^+}}+\frac{1}{\epsilon^{sp^-}}\right)\left(\||u|^{p^+}\|_{L^{\frac{(p^+)_s^*}{p^+}}(B_{2\epsilon}(0))}+\||u|^{p^-}\|_{L^{\frac{(p^+)_s^*}{p^+}}(B_{2\epsilon}(0))}\right).\nonumber
		\end{align}
	Hence, $I\rightarrow 0$ as $\epsilon\rightarrow 0.$	Similarly the second term in the right hand side of $\eqref{2}$ can be rewritten as follows.
		\begin{align}
			II&=\sum_{k=1}^{\infty}(\epsilon^N+\epsilon^{N+s(p^--p^+)}) \int_{2^k\epsilon\leq|x|\leq2^{k+1}\epsilon}\frac{|u(x)|^{p^+}+|u(x)|^{p^-}}{|x|^{N+sp^-}}dx\nonumber\\& \leq \sum_{k=1}^{\infty}\frac{1}{2^{k(N+sp^-)}}(1/\epsilon^{sp^-}+1/\epsilon^{sp^+}) \int_{|x|\leq2^{k+1}\epsilon}{|u(x)|^{p^+}+|u(x)|^{p^-}}dx\nonumber\\&\leq\sum_{k=1}^{\infty}\frac{C\epsilon^{sp^+}}{2^{k(N+sp^--sp^+)}}(1/\epsilon^{sp^-}+1/\epsilon^{sp^+})\left(\||u|^{p^+}\|_{L^{\frac{(p^+)_s^*}{p^+}}(B_{2^{k+1}\epsilon}(0))}+\||u|^{p^-}\|_{L^{\frac{(p^+)_s^*}{p^+}}(B_{2^{k+1}\epsilon}(0))}\right)\nonumber\\&\leq\sum_{k=1}^{\infty}\frac{C\epsilon^{sp^+}}{2^{ksp^-}}(1/\epsilon^{sp^-}+1/\epsilon^{sp^+})\left(\||u|^{p^+}\|_{L^{\frac{(p^+)_s^*}{p^+}}(B_{2^{k+1}\epsilon}(0))}+\||u|^{p^-}\|_{L^{\frac{(p^+)_s^*}{p^+}}(B_{2^{k+1}\epsilon}(0))}\right)
		\end{align}
		Now for any $\gamma>0$, there exists a $k'\in \mathbb{N}$ such that $\sum_{k=k'+1}^{\infty}2^{-ksp^-}<\gamma$. So,
		\begin{align}\label{b}
			II&\leq \gamma C\epsilon^{sp^+}\left(\frac{1}{\epsilon^{sp^+}}+\frac{1}{\epsilon^{sp^-}}\right)\left(\||u|^{p^+}\|_{L^{\frac{(p^+)_s^*}{p^+}}(\mathbb{R}^N)}+\||u|^{p^-}\|_{L^{\frac{(p^+)_s^*}{p^+}}(\mathbb{R}^N)}\right)\nonumber\\&~~+C\epsilon^{sp^+}\left(\frac{1}{\epsilon^{sp^+}}+\frac{1}{\epsilon^{sp^-}}\right)\sum_{k=1}^{k'}\frac{1}{2^{ksp^-}}\left(\||u|^{p^+}\|_{L^{\frac{(p^+)_s^*}{p^+}}(B_{2^{k'+1}\epsilon}(0))}+\||u|^{p^-}\|_{L^{\frac{(p^+)_s^*}{p^+}}(B_{2^{k'+1}\epsilon}(0))}\right).
		\end{align} 
		Hence, $II\rightarrow0$ as $\epsilon\rightarrow0$. Hence, the claim. Therefore, from the inequality $\eqref{aa}$ we establish
	$$\lim_{\epsilon\rightarrow 0}\lim_{n\rightarrow \infty}\int_{\mathbb{R}^{2N}}\frac{|u_n(x)-u_n(y)|^{p(x,y)-2}(u_n(x)-u_n(y))u_n(y)(\phi_{\epsilon,j}(x)-\phi_{\epsilon,j}(y))}{|x-y|^{N+sp(x,y)}}dxdy\rightarrow 0.$$
 Following the argument used in the proof of $\eqref{a}$ and using $\eqref{claim}$ we find 
 \begin{equation}
\lim_{\epsilon\rightarrow 0}\lim_{n\rightarrow \infty} \int_{\mathbb{R}^{2N}}\frac{|u_n(x)-u_n(y)|^{p^+-2}(u_n(x)-u_n(y))u_n(y)(\phi_{\epsilon,j}(x)-\phi_{\epsilon,j}(y))}{|x-y|^{N+sp^+}}dxdy \rightarrow 0.\nonumber
 \end{equation}
Applying $(\mathcal{F}_2)$ and passing the limit $n\rightarrow\infty$ in the inequality $\eqref{a1}$ we get $$0=\int_{\mathbb{R}^{N}}\phi_{\epsilon,j}d\mu+\int_{\mathbb{R}^{N}}\phi_{\epsilon,j}d\sigma-\int_{\Omega}\phi_{\epsilon,j}d\nu-\lambda\int_{\Omega}f(x,u)u\phi_{\epsilon,j}dx.$$
Since $\phi_{\epsilon,j}(x)\rightarrow 0$ as $\epsilon\rightarrow0$ for any $x\neq x_j$ and $\phi(0)=1$, we have $$\lim_{\epsilon\rightarrow 0}\int_{\Omega}f(x,u)u\phi_{\epsilon,j}dx\rightarrow0,~\lim_{\epsilon\rightarrow 0}\int_{\Omega}\phi_{\epsilon,j}d\nu=\nu_j,$$ $$\lim_{\epsilon\rightarrow 0}\int_{\mathbb{R}^{N}}\phi_{\epsilon,j}d\mu=\mu_j,~\lim_{\epsilon\rightarrow 0}\int_{\mathbb{R}^{N}}\phi_{\epsilon,j}d\sigma=\sigma_j$$
where $\nu_j=\nu(\{x_j\})$, $\mu_j=\mu(\{x_j\})$ and $\sigma_j=\sigma(\{x_j\})$.	Hence, $\mu_i+\sigma_i=\nu_i$ for every $i\in I$. This implies $\mu_i\leq\nu_i$ and from $\eqref{cc4}$ we get
	\begin{equation}
\nu_i\leq2\max\left\{\frac{\nu_j^{\frac{(p^+)_s^*}{p^+}}}{S_{p^+,r,p_s^*}^{(p^+)_s^*}},\frac{\nu_j^{\frac{m}{p^+}}}{S_{p^+,r,p_s^*}^{m}}\right\}, ~\forall i\in I.\nonumber
	\end{equation}
	This implies, either $\nu_i=0$ or $\min\left\{2^{\frac{sp^+-N}{sp^+}}S_{p^+,r,p_s^*}^{\frac{N}{s}},2^{\frac{p^+}{p^+-m}}S_{p^+,r,p_s^*}^{\frac{mp^+}{m-p^+}} \right\}\leq \nu_i.$\\
Denote $m=\min\{r^-,(p^*_s)^-\}$ and $A_\gamma=\underset{x\in A=A_1\cup A_2}{\cup}(B_\gamma(x)\cap\Omega)$ for some $\gamma>0$. Then, by $(\mathcal{F}_1)$ and $(\mathcal{F}_4)$, we obtain
	\begin{align}\label{estimate}
	c&=\lim_{n\rightarrow \infty}\Psi(u_n)\nonumber\\&=\lim_{n\rightarrow \infty}\left(\Psi(u_n)-\frac{1}{p^+}\langle\Psi^\prime(u_n),u_n\rangle\right)\nonumber\\&=\lim_{n\rightarrow \infty}
	\left(\int_{\mathbb{R}^{2N}}\left(\frac{1}{p(x,y)}-\frac{1}{p^+}\right)\frac{|u_n(x)-u_n(y)|^{p(x,y)}}{|x-y|^{N+sp(x,y)}}dxdy+\int_{\Omega}\left(\frac{1}{p^+}-\frac{1}{p_s^*(x)}\right)|u_n|^{p_s^*(x)}dx\right)\nonumber\\&~~~~+\lim_{n\rightarrow \infty}\int_{\Omega}\left(\frac{1}{p^+}-\frac{1}{r(x)}\right)|u_n|^{r(x)}dx-\lambda \lim_{n\rightarrow \infty}\int_{\Omega}\left(F(x,u_n)-\frac{1}{p^+}f(x,u_n)u_n\right)dx\nonumber\\& \geq \lim_{n\rightarrow \infty}\left(\frac{1}{p^+}-\frac{1}{m}\right)\int_{A_\gamma}\left(|u_n|^{r(x)}+|u_n|^{p^*_s(x)}\right) dx-\frac{a}{m}\lambda \int_{\Omega}|u|^{\alpha(x)}dx\nonumber\\
	&=\left(\frac{1}{p^+}-\frac{1}{m}\right)\left(\int_{A_\gamma}\left(|u|^{r(x)}+|u|^{p^*_s(x)}\right)dx+\sum_{i\in I}\nu_i\right)-\frac{a}{m}\lambda\max\{\|u\|_{L^{\alpha(\cdot)}(\Omega)}^{\alpha^+},\|u\|_{L^{\alpha(\cdot)}(\Omega)}^{\alpha^-}\}\nonumber\\&\geq\left(\frac{1}{p^+}-\frac{1}{m}\right)\nu_j+\left(\frac{1}{p^+}-\frac{1}{m}\right)\|u\|_{L^{r(\cdot)}(\Omega)}^{\bar{r}}-\frac{a}{m}\lambda\|u\|_{L^{\alpha(\cdot)}(\Omega)}^{\bar{\alpha}}\nonumber\\&\geq\left(\frac{1}{p^+}-\frac{1}{m}\right)\min\left\{2^{\frac{sp^+-N}{sp^+}}S_{p^+,r,p_s^*}^{\frac{N}{s}},2^{\frac{p^+}{p^+-m}}S_{p^+,r,p_s^*}^{\frac{mp^+}{m-p^+}} \right\}+\left(\frac{1}{p^+}-\frac{1}{m}\right)\|u\|_{L^{r(\cdot)}(\Omega)}^{\bar{r}}-C_0\lambda\|u\|_{L^{r(\cdot)}(\Omega)}^{\bar{\alpha}}.
	\end{align}
	Here $\|u\|_{L^{r(\cdot)}(\Omega)}^{\bar{r}}=\min\{\|u\|_{L^{r(\cdot)}(\Omega)}^{r^+},\|u\|_{L^{r(\cdot)}(\Omega)}^{r^-}\}$, $\|u\|_{L^{\alpha(\cdot)}(\Omega)}^{\bar{\alpha}}=\max\{\|u\|_{L^{\alpha(\cdot)}(\Omega)}^{\alpha^+},\|u\|_{L^{\alpha(\cdot)}(\Omega)}^{\alpha^-}\}$ and $C_0=\frac{a}{m}\left[2(1+|\Omega|)\right]^{\bar{\alpha}}$ (see Proposition $\ref{holder}$). Clearly, $\bar{\alpha}<\bar{r}$. Let us consider the function $h:(0,\infty)\rightarrow\mathbb{R}$ by
	$$h(x)=\left(\frac{1}{p^+}-\frac{1}{m}\right)x^{\bar
		r}-C_0\lambda x^{\bar{\alpha}}.$$
Thus, $h$ attains its minimum at $\bar{x}=\left(\frac{\lambda C_0\bar{\alpha}mp^+}{\bar{r}(m-p^+)}\right)^{\frac{1}{\bar{r}-\bar{\alpha}}}$ and $$h(x)\geq h(\bar{x})=-C_1\lambda^{C_2}$$
where $C_1=C_0\left(\frac{\bar{r}-\bar{\alpha}}{\bar{r}}\right)\left(\frac{ C_0\bar{\alpha}mp^+}{\bar{r}(m-p^+)}\right)^{\frac{\bar{\alpha}}{\bar{r}-\bar{\alpha}}}>0$, $C_2=\frac{\bar{r}}{\bar{r}-\bar{\alpha}}$.	Therefore, from $\eqref{estimate}$ we obtain
	$$c\geq \left(\frac{1}{p^+}-\frac{1}{m}\right)\min\left\{2^{\frac{sp^+-N}{sp^+}}S_{p^+,r,p_s^*}^{\frac{N}{s}},2^{\frac{p^+}{p^+-m}}S_{p^+,r,p_s^*}^{\frac{mp^+}{m-p^+}} \right\}-C_1\lambda^{C_2}.$$
	This implies the indexing set $I=\emptyset$ if we are to have    $$c<\left(\frac{1}{p^+}-\frac{1}{m}\right)\min\left\{2^{\frac{sp^+-N}{sp^+}}S_{p^+,r,p_s^*}^{\frac{N}{s}},2^{\frac{p^+-m}{p^+}}S_{p^+,r,p_s^*}^{\frac{mp^+}{m-p^+}} \right\}-C_1\lambda^{C_2}.$$ 
	Hence, $|u_n|^{r(\cdot)}\overset{t}{\rightharpoonup}|u|^{r(\cdot)}$ and $|u_n|^{p_s^*(\cdot)}\overset{t}{\rightharpoonup}|u|^{p_s^*(\cdot)}$. Therefore, using Prokhorov's theorem we have $u_n\rightarrow u$ strongly in $L^{r(\cdot)}(\Omega)$ and in $L^{p_s^*(\cdot)}(\Omega)$.\\
	Define
	$$\langle I_1(u),v\rangle=\int_{\mathbb{R}^{2N}}\frac{|u(x)-u(y)|^{p(x,y)-2}(u(x)-u(y))(v(x)-v(y))}{|x-y|^{N+sp(x,y)}}dxdy$$ and $$\langle I_2(u),v\rangle=\int_{\mathbb{R}^{2N}}\frac{|u(x)-u(y)|^{p^+-2}(u(x)-u(y))(v(x)-v(y))}{|x-y|^{N+sp^+}}dxdy.$$
Observe
	\begin{align}\label{s}
	\langle\Psi^\prime(u_n),(u_n-u)\rangle&=\langle I_1(u_n),(u_n-u)\rangle+\langle I_2(u_n),(u_n-u)\rangle-\int_{\Omega}|u_n|^{r(x)-2}u_n(u_n-u)dx\nonumber\\&~~~~-\int_{\Omega}|u_n|^{p_s^*(x)-2}u_n(u_n-u)dx-\lambda\int_{\Omega}f(x,u_n)(u_n-u)dx.
	\end{align}
	Since $\{u_n\}$ is bounded in $X$, $\{u_n\}$ is also bounded in $L^{r(\cdot)}(\Omega)$ and  $L^{p_s^*(\cdot)}(\Omega)$. Thus, on applying the H\"{o}lder's inequality we get
	\begin{align}
	\left|\int_{\Omega}|u_n|^{r(x)-2}u_n(u_n-u)dx\right|&\leq \int_{\Omega}|u_n|^{r(x)-1}|u_n-u|dx\nonumber\\&\leq C \||u_n|^{r(\cdot)-1}\|_{L^{\frac{r(\cdot)}{r(\cdot)-1}}(\Omega)}\|u_n-u\|_{L^{r(\cdot)}(\Omega)}\nonumber\\&\leq C\|u_n-u\|_{L^{r(\cdot)}(\Omega)}=o_n(1).
	\end{align}
	Similarly we can show that $$\lim_{n\rightarrow \infty}\left(\int_{\Omega}\lambda f(x,u_n)(u_n-u)dx+\int_{\Omega}|u_n|^{p_s^*(x)-2}u_n(u_n-u)dx\right)\rightarrow 0.$$ Passing the limit $n\rightarrow \infty$ in $\eqref{s}$ we get
	$\langle I_1(u_n),(u_n-u)\rangle+\langle I_2(u_n),(u_n-u)\rangle\rightarrow 0$. Therefore,
	\begin{equation}\label{s+}
	\lim_{n\rightarrow \infty}\left(\langle I_1(u_n)-I_1(u),(u_n-u)\rangle+\langle I_2(u_n)-I_2(u),(u_n-u)\rangle\right)=0.
	\end{equation}
		Recall the Simon's inequality \cite{Simon} given as
		\begin{align}
		|x-y|^p&\leq\frac{1}{p-1}\left[\left(|x|^{p-2}x-|y|^{p-2}y\right).(x-y)\right]^{\frac{p}{2}}\left(|x|^p+|y|^p\right)^{\frac{2-p}{2}},~\text{if}~1<p<2.\nonumber\\|x-y|^p&\leq 2^p \left(|x|^{p-2}x-|y|^{p-2}y\right).(x-y),,~\text{if}~p\geq2.
		\end{align}
Let us first consider the case $p^+>2$. Then, using the Simon's inequality we have 
\begin{align}\label{G}
\|u_n-u\|_{s,p^+}^{p^+}&=\int_{\mathbb{R}^{2N}}\frac{|(u_n-u)(x)-(u_n-u)(y)|^{p^+}}{|x-y|^{N+sp^+}}dxdy\nonumber\\&\leq\frac{1}{(p^+-1)}\int_{\mathbb{R}^{2N}}
\left\{\frac{|u_n(x)-u_n(y)|^{p^+-2}(u_n(x)-u_n(y))((u_n-u)(x)-(u_n-u)(y))}{|x-y|^{N+sp^+}}\right. \nonumber \\
&~~~~\left.\kern-\nulldelimiterspace -\;\frac{|u(x)-u(y)|^{p^+-2}(u(x)-u(y))((u_n-u)(x)-(u_n-u)(y))}{|x-y|^{N+sp^+}}\vphantom{}\right\}\nonumber\\&\leq C_1 \langle I_2(u_n)-I_2(u),(u_n-u)\rangle.
\end{align}
Similarly for $1<p^+<2$, using H\"{o}lder's inequality and the boundedness of $\{u_n\}$ in $W_0^{s,p^+}(\Omega)$, we establish the following.
\begin{align}\label{L}
\|u_n-u\|_{s,p^+}^{p^+}&\leq 2^{p^+}\langle I_2(u_n)-I_2(u),(u_n-u)\rangle^{\frac{p^+}{2}}\left(\|u_n\|_{s,p^+}^{p^+}+\|u\|_{s,p^+}^{p^+}\right)^{\frac{2-p^+}{2}}\nonumber\\& \leq C_2\langle I_2(u_n)-I_2(u),(u_n-u)\rangle^{\frac{p^+}{2}}\left(\|u_n\|_{s,p^+}^{p^+\frac{2-p^+}{2}}+\|u\|_{s,p^+}^{p^+\frac{2-p^+}{2}}\right)\nonumber\\&\leq C_3\langle I_2(u_n)-I_2(u),(u_n-u)\rangle^{\frac{p^+}{2}}.
\end{align}
On combining the inequalities $\eqref{s+}$, $\eqref{G}$ and $\eqref{L}$, we obtain
$$\lim_{n\rightarrow \infty}\langle I_1(u_n)-I_1(u),(u_n-u)\rangle\leq0.$$
	Since $I_1$ is of $(S_+)$-type by Lemma $\ref{bah}$ (refer Appendix), we conclude that $u_n\rightarrow u$ strongly in $W_0$ and hence by simple calculation we get $\lim_{n\rightarrow \infty}\langle I_1(u_n)-I_1(u),(u_n-u)\rangle=0$. Therefore, from $\eqref{s+}$, $\lim_{n\rightarrow \infty}\langle I_2(u_n)-I_2(u),(u_n-u)\rangle=0$. By $\eqref{G}$ and $\eqref{L}$ we obtain that $u_n\rightarrow u$ strongly in $W_0^{s,p^+}(\Omega)$ and hence $u_n\rightarrow u$ strongly in X.
\end{proof}
\noindent The functional $\Psi$ does not satisfy the hypotheses of Lemma $\ref{symmetric}$, since $\Psi$ is not bounded from below. Thus, we now introduce a truncated functional related to $\Psi$ that verifies the assumptions of Lemma $\ref{symmetric}$.\\
Let $\lambda_1$ be the first eigenvalue of $(-\Delta)^s_{p^+}$ (see \cite{Lindgren}), i.e
$$\lambda_1=\underset{u\in W_0^{sp^+}(\Omega)\setminus\{0\}}{\inf}\frac{\|u\|_{s,p^+}^{p^+}}{\|u\|_{L^{p^+}(\Omega)}^{p^+}}.$$
Then, for any $\lambda\in (0,\lambda_1)$, using Proposition $\ref{holder}$, Theorem $\ref{cap}$, $\eqref{best constant}$ and $(\mathcal{F}_2)$ we observe
	\begin{align}
	\Psi(u)&=\int_{\mathbb{R}^N}\int_{\mathbb{R}^N}\frac{1}{p(x,y)}\frac{|u(x)-u(y)|^{p(x,y)}}{|x-y|^{N+sp(x,y)}}dydx+\frac{1}{p^+}\int_{\mathbb{R}^N}\int_{\mathbb{R}^N}\frac{|u(x)-u(y)|^{p^+}}{|x-y|^{N+sp^+}}dydx\nonumber\\&~~~~-\int_{\Omega}\frac{1}{r(x)}|u|^{r(x)}dx-\int_{\Omega}\frac{1}{p_s^*(x)}|u|^{p_s^*(x)}dx-\lambda\int_{\Omega}F(x,u)dx\nonumber\\&\geq \frac{1}{p^+}\|u\|_{s,p^+}^{p^+}-\frac{1}{r^-}\left(\frac{\|u\|_{s,p^+}}{S_{p^+,r,p_s^*}}\right)^{r^\pm}-\frac{1}{(p^*)^-}\left(\frac{\|u\|_{s,p^+}}{S_{p^+,r,p_s^*}}\right)^{(p_s^*)^\pm}\nonumber\\&~~~~-\frac{\lambda_1C}{\beta^-}\left(\frac{2(1+|\Omega|)\|u\|_{s,p^+}}{S_{p^+,r,p_s^*}}\right)^{\beta^\pm}-2\lambda C(1+|\Omega|)\frac{\|u\|_{s,p^+}}{S_{p^+,r,p_s^*}}.\nonumber
	\end{align}
Here $U^{\alpha^\pm}$ denotes $\max\{U^{\alpha^+}, U^{\alpha^-}\}$. Thus, we have 
$$\Psi(u)\geq A\|u\|_{s,p^+}^{p^+}-B\|u\|_{s,p^+}^{r^\pm}-C\|u\|_{s,p^+}^{(p_s^*)^\pm}-D\|u\|_{s,p^+}^{\beta^\pm}-\lambda E\|u\|_{s,p^+},$$
where $$A=\frac{1}{p^+},~B=\frac{S_{p^+,r,p_s^*}^{-r^\pm}}{r^-},~C=\frac{S_{p^+,r,p_s^*}^{-(p_s^*)^\pm}}{(p_s^*)^-},~D=\frac{\lambda_1C}{\beta^-}\left(\frac{2(1+|\Omega|)}{S_{p^+,r,p_s^*}}\right)^{\beta^\pm},~E=2 C(1+|\Omega|)S_{p^+,r,p_s^*}^{-1}.$$
Let us define a function $g:(0,\infty)\rightarrow\mathbb{R}$ by
$$g(x)=Ax^{p^+}-Bx^{r^\pm}-Cx^{(p_s^*)^\pm}-Dx^{\beta^\pm}-\lambda Ex.$$
We have $1<p^+<\beta^-\leq \beta^+<m=\min\{(p_s^*)^-,r^-\}$. Therefore, we may choose $\lambda_1\geq \Lambda>0$, very small, such that for every $\lambda\in (0,\Lambda)$ function $g$ has finitely many positive roots, say $0<r_1<r_2<\cdot\cdot\cdot<r_m<\infty$ and
\begin{enumerate}
	\item[(i)] $g$ attains its maximum with $\underset{x\in(0,\infty)}{\max}g(x)>0$,
	\item[(ii)] $\left(\frac{1}{p^+}-\frac{1}{m}\right)\min\left\{2^{\frac{sp^+-N}{sp^+}}S_{p^+,r,p_s^*}^{\frac{N}{s}},2^{\frac{p^+}{p^+-m}}S_{p^+,r,p_s^*}^{\frac{mp^+}{m-p^+}} \right\}-C_1\lambda^{C_2}>0$ \\where $C_1,C_2$ are given in Lemma $\ref{palais}$.
\end{enumerate}
More precisely,
\begin{eqnarray}\label{g}
\begin{split}
g(x)&\begin{cases}
<0 ~\forall~ x\in A=(0,r_1)\cup(r_2,r_3)\cup\cdot\cdot\cdot\cup(r_m,\infty)\\
>0 ~\forall~x\in B= (r_1,r_2)\cup\cdot\cdot\cdot\cup(r_{m-1},r_m).
\end{cases}
\end{split}
\end{eqnarray} 
Let us consider the truncated functional $\bar{\Psi}:X\rightarrow\mathbb{R}$ defined as
\begin{align}
\bar{\Psi}(u)&=\iint_{\mathbb{R}^{2N}}\frac{1}{p(x,y)}\frac{|u(x)-u(y)|^{p(x,y)}}{|x-y|^{N+sp(x,y)}}dydx+\frac{1}{p^+}\iint_{\mathbb{R}^{2N}}\frac{|u(x)-u(y)|^{p^+}}{|x-y|^{N+sp^+}}dydx\nonumber\\&~~~~-\tau(\|u\|_{s,p^+})\int_{\Omega}\frac{1}{r(x)}|u|^{r(x)}dx-\tau(\|u\|_{s,p^+})\int_{\Omega}\frac{1}{p_s^*(x)}|u|^{p_s^*(x)}dx-\lambda \tau(\|u\|_{s,p^+})\int_{\Omega}F(x,u)dx,\nonumber
\end{align}
where $\tau\in C^\infty(\mathbb{R}^+:[0,1])$ such that
\begin{eqnarray}\label{g1}
\begin{split}
\tau(x)=&\begin{cases}
 1~\text{if}~ x\in A\setminus(r_m,\infty)\\
0 ~\text{if}~x\in (r_m,\infty).
\end{cases}
\end{split}
\end{eqnarray}
Hence, we have 
\begin{align}\label{g2}
\bar{\Psi}(u)&\geq A\|u\|_{s,p^+}^{p^+}-B\tau(\|u\|_{s,p^+})\|u\|_{s,p^+}^{r^\pm}-C\tau(\|u\|_{s,p^+})\|u\|_{s,p^+}^{(p_s^*)^\pm}\nonumber\\&~~~~-D\tau(\|u\|_{s,p^+})\|u\|_{s,p^+}^{\beta^\pm}-\lambda E\tau(\|u\|_{s,p^+})\|u\|_{s,p^+}\nonumber\\&:= \bar{g}(\|u\|_{s,p^+}),
\end{align}
where $\bar{g}(x)=Ax^{p^+}-B\tau(x)x^{r^\pm}-C\tau(x)x^{(p_s^*)^\pm}-D\tau(x)x^{\beta^\pm}-\lambda E\tau(x)x.$ It is simple to check that
\begin{eqnarray}\label{g3}
\begin{split}
\bar{g}(x)&\begin{cases}
\geq g(x)~\forall~x>0\\
= g(x) ~\forall~x\in A\setminus(r_m,\infty)\\
\geq 0~\forall~x>r_m
\end{cases}
\end{split}
\end{eqnarray}
and thus 
\begin{equation}\label{g4}
\Psi(u)=\bar{\Psi}(u), ~\text{for}~\|u\|_{s,p^+}\in A\setminus(r_m\infty).
\end{equation}
By using $\eqref{g}-\eqref{g4}$ and Lemma $\ref{palais}$, it is not difficult to verify the following properties for the functional $\bar{\Psi}$.
\begin{lemma}\label{l1}
	\begin{enumerate}
		\item $\bar{\Psi}\in C^1(X,\mathbb{R})$, $\bar{\Psi}$ is even and bounded from below.
		\item $\|u\|_{s,p^+}\in A\setminus(r_m,\infty)$ whenever $\bar{\Psi}(u)<0$ and there exists a neighbourhood $N_u$ of $u$ such that $\Psi(v)=\bar{\Psi}(v)$ for $v\in N_u$.
		\item Consider $\Lambda>0$ as given above such that $(i)$ and $(ii)$ hold. Then, for any $\lambda\in (0,\Lambda)$, the functional $\bar{\Psi}$ satisfies the Palais-Smale (P-S) condition for every $c<0$. 
	\end{enumerate}
\end{lemma}
\begin{lemma}\label{l2}
Let $(\mathcal{F}_3)$ holds. Then, for every $k\in\mathbb{R}$, there exists $\epsilon(k)>0$ such that $\gamma\left(\{u\in X:\bar{\Psi}(u)\leq -\epsilon(k)\}\setminus\{0\}\right)\geq k.$
	\begin{proof}
	Let $\lambda\in(0,\Lambda)$ and $k\in \mathbb{N}$. Let $Y^k$ be a $k$-dimensional subspace of $X$ and consider $u\in Y^k$ with $\|u\|_X=1$. Thus, for $0<\delta<r_1$ we have $\delta\|u\|_{s,p^+}\leq \delta\|u\|_X<r_1$ and $\tau(\delta\|u\|_{s,p^+})=1$. According to the assumption $(\mathcal{F}_3)$, for any $u\in X\setminus\{0\}$, it holds
	$$F(x,\delta u)\geq H(\delta)(\delta u)^{p^-}~\text{with}~H(\delta)\rightarrow\infty ~\text{as}~\delta\rightarrow 0.$$
	As $Y^k$ is a finite dimensional subspace, all
	the norms in $Y^k$ are equivalent. We now define
	$$\eta_k=\inf\{\|u\|_{L^{p^-}(\Omega)}^{p^-}:u\in Y^k,\|u\|_X=1\}>0.$$
	Thus, for any $u\in Y^k$ with $\|u\|_X=1$ and $\delta\in(0,r_1)$ we have 
	\begin{align}
	\Psi(\delta u)=\bar{\Psi}(\delta u)&\leq \frac{1}{p^+}\delta^{p^+}+\frac{1}{p^-}\delta^{p^\pm}-\lambda H(\delta)\delta^{p^-}\eta_k\nonumber\\&= \left[ \frac{1}{p^+}\delta^{p^+-p^-}+\frac{1}{p^-}\delta^{p^\pm-p^-}-\lambda H(\delta)\eta_k\right]\delta^{p^-}.\nonumber
	\end{align}
	Since $H(\delta)\rightarrow\infty$ as $\delta\rightarrow0$, for any $\epsilon(k)>0$ there exists $\delta(0,r_1)$ such that $\bar{\Psi}(\delta u)\leq -\epsilon(k)$. This implies
	$$A_k=\{u\in Y^k:\|u\|_X=\delta\}\subset\{u\in X:\bar{\Psi}(u)\leq -\epsilon(k)\}\setminus\{0\}$$
	and by Proposition $\ref{properties}$, $\gamma\left(\{u\in X:\bar{\Psi}(u)\leq -\epsilon(k)\}\setminus\{0\}\right)\geq \gamma(A_k\cap Y^k) \geq k.$
\end{proof}
\end{lemma}
\noindent We are now ready to prove our main result, i.e. Theorem $\ref{exist}$.
\begin{proof}[Proof of Theorem $\ref{exist}$]
Let us recall that 
$$\Gamma_k=\{A\subset X\setminus\{0\}: A\text{ is closed and symmetric},~0\notin A~\text{and}~\gamma(A)\geq k\}$$
and denote
$$c_k=\underset{A\in \Gamma_k}{\inf}~\underset{u\in A}{\sup}~\bar{\Psi}(u).$$
In the view of Lemma $\ref{l1}$ $(1)$ and Lemma $\ref{l2}$, we deduce that $-\infty<c_k<0$. As $\Gamma_{k+1}\subset \Gamma_k$, is it true that $c_k\leq c_{k+1}<0$ for any $k\in\mathbb{N}$. Thus, $\lim\limits_{k\rightarrow\infty}c_k=c\leq 0$. Following the arguments used in \cite{Rabinowitz}, we conclude that $c=0$ and each $c_k$ is a critical value of the functional $\bar{\Psi}$. Therefore, from Lemma $\ref{palais}$, Lemma $\ref{l1}$ and Lemma $\ref{l2}$, we insure that $\bar{\Psi}$ verifies the hypotheses $(a)$ and $(b)$ of Lemma $\ref{symmetric}$. Finally, with the consideration of Lemma $\ref{l1}$ $(2)$, we guarantee the existence of a sequence of weak solutions $\{u_k\}\subset X$ to $\eqref{app}$ that converges to 0 and hence the proof. 
\end{proof}
\section{Appendix}\label{Appendix}
Following are a few lemmas and results that have been used at several places in the manuscript.
\begin{corollary}[\cite{Bonder1}]\label{Linfinity p^+}
	Let $\phi\in W^{1,\infty}(\mathbb{R}^N)$ such that support of $\phi$ lies in the unit ball of $\mathbb{R}^N$ and given $\epsilon>0$, $x_0\in\mathbb{R}^N$ define $\phi_{\epsilon,x_0}(x)=\phi(\frac{x-x_0}{\epsilon})$. Then,
	$$\int_{\mathbb{R}^N}\frac{|\phi_{\epsilon,x_0}(x)-\phi_{\epsilon,x_0}(y)|^{p}}{|x-y|^{N+sp}}dy\leq C\min\left(\epsilon^{-sp}, \epsilon^N|x-x_0|^{-(N+sp)}\right)$$ where $C$ depends on $N,s,p,\|\phi\|_{1,\infty}$.
\end{corollary}
\noindent The following lemma provides the decay estimate and the scaling property of compactly supported nonlocal gradient of smooth functions.
\begin{lemma}\label{lemma1}
	Let $1<p^-\leq p(x,y)\leq p^+<\infty$ for every $(x,y)\in \mathbb{R}^N\times\mathbb{R}^N$, $sp^+<N$, $\phi\in C_c^\infty(\mathbb{R}^N)$ such that $0\leq\phi\leq1,~\phi(0)=1$  and support of $\phi$ lies in the unit ball of $\mathbb{R}^N$. For some $x_0\in\mathbb{R}^N$ and $\epsilon>0$, define $\phi_{\epsilon,x_0}(x)=\phi(\frac{x-x_0}{\epsilon})$. Then, $$\int_{\mathbb{R}^N}\frac{|\phi_{\epsilon,x_0}(x)-\phi_{\epsilon,x_0}(y)|^{p(x,y)}}{|x-y|^{N+sp(x,y)}}dy\leq C\min\left(1/\epsilon^{sp^+}+1/\epsilon^{sp^-}, (\epsilon^N+ \epsilon^{N+s(p^--p^+)})|x-x_0|^{-(N+sp^-)}\right)$$where $C$ depends on $N,s,p,\|\phi\|_{1,\infty}$.
\end{lemma}
\begin{proof}
	We first observe that
	\begin{align}\label{esti1}
	\int_{\mathbb{R}^N}\frac{|\phi_{\epsilon,x_0}(x)-\phi_{\epsilon,x_0}(y)|^{p(x,y)}}{|x-y|^{N+sp(x,y)}}dy&=\int_{\mathbb{R}^N}\frac{|\phi{(\frac{x-x_0}{\epsilon})}-\phi(\frac{y-x_0}{\epsilon})|^{p(x,y)}}{|x-y|^{N+sp(x,y)}}dy\nonumber\\& \leq\left(1/\epsilon^{sp^+}+1/\epsilon^{sp^-}\right)\int_{\mathbb{R}^N}\frac{|\phi(x')-\phi(y')|^{p(x'+\epsilon x_0,y'+\epsilon x_0)}}{|x'-y'|^{N+sp(x'+\epsilon x_0,y'+\epsilon x_0)}}dy'.
	\end{align}
	Denote $\tilde{p}(x',y')=p(x'+\epsilon x_0,y'+\epsilon x_0)$ and decompose the integral on the right hand side of $\eqref{esti1}$ as follows.
	\begin{align}
	\int_{\mathbb{R}^N}\frac{|\phi(x')-\phi(y')|^{\tilde{p}(x',y')}}{|x'-y'|^{N+s\tilde{p}(x',y')}}dy'\nonumber&= \left(\int_{|x'-y'|\geq1}+\int_{|x'-y'|<1}\right)\frac{|\phi(x')-\phi(y')|^{\tilde{p}(x',y')}}{|x'-y'|^{N+s\tilde{p}(x',y')}}dy'\nonumber\\& =I+II.\nonumber
	\end{align}
	We try to find $L^\infty$ bounds of these two integrals.
	\begin{align}\label{esti2}
	I &= \int_{|x'-y'|\geq1}\frac{|\phi(x')-\phi(y')|^{\tilde{p}(x',y')}}{|x'-y'|^{N+s\tilde{p}(x',y')}}dy'\nonumber\\& \leq C\|\phi\|_{\infty}^{p^-}\int_{|x'-y'|\geq1}\frac{1}{|x'-y'|^{N+sp^-}}dy'\nonumber\\&\leq C\frac{\|\phi\|_\infty^{p^-}}{sp^-}
	\end{align}
	and 
	\begin{align}\label{esti3}
	II&= \int_{|x'-y'|<1}\frac{|\phi(x')-\phi(y')|^{\tilde{p}(x',y')}}{|x'-y'|^{N+s\tilde{p}(x',y')}}dy'\nonumber
	\\&\leq \|\nabla\phi\|_{\infty}^{p^\pm} \int_{|x'-y'|<1}\frac{1}{|x'-y'|^{N+s\tilde{p}(x',y')-\tilde{p}(x',y')}}dy'\nonumber\\& \leq \|\nabla\phi\|_{\infty}^{p^\pm} \int_{|x'-y'|<1}\frac{1}{|x'-y'|^{N-p^-(1-s)}}dy'\nonumber\\& \leq C\frac{\|\nabla\phi\|_\infty^{p^\pm}}{p^-(1-s)}.
	\end{align}
	In order to obtain a decay estimate, we restrict ourselves to the case where $|x'|>2$ such that $\phi(x')=0$. Hence, we observe that $|x'-y'|\geq|x'|-1\geq\frac{|x'|}{2}$ and 
	\begin{align}\label{esti4}
	\int_{\mathbb{R}^N}\frac{|\phi(x')-\phi(y')|^{\tilde{p}(x',y')}}{|x'-y'|^{N+s\tilde{p}(x',y')}}dy'& = \int_{\mathbb{R}^N}\frac{|\phi(y')|^{\tilde{p}(x',y')}}{|x'-y'|^{N+s\tilde{p}(x',y')}}dy'\nonumber\\&\leq \int_{|y'|\leq 1}\frac{|\phi(y')|^{\tilde{p}(x',y')}2^{N+s\tilde{p}(x',y')}}{|x'|^{N+s\tilde{p}(x',y')}}dy'\nonumber\\& \leq\frac{C\|\phi\|_\infty^{p^-}}{|x'|^{N+sp^-}}.
	\end{align}
	Combining $\eqref{esti1},\eqref{esti2},\eqref{esti3}$  and $\eqref{esti4}$ we get
	\begin{align}
	\int_{\mathbb{R}^N}\frac{|\phi_{\epsilon,x_0}(x)-\phi_{\epsilon,x_0}(y)|^{p(x,y)}}{|x-y|^{N+sp(x,y)}}dy&\leq C(1/\epsilon^{sp^+}+1/\epsilon^{sp^-})\min\left(1,\Big|\frac{x-x_0}{\epsilon}\Big|^{-(N+sp^-)}\right)\nonumber\\&\leq C\min\left(1/\epsilon^{sp^+}+1/\epsilon^{sp^-}, (\epsilon^N+ \epsilon^{N+s(p^--p^+)})|x-x_0|^{-(N+sp^-)}\right)\nonumber
	\end{align}
	where $C > 0$ depends on $N, s, p$ and $\|\phi\|_\infty$.
\end{proof}
\begin{lemma}[\cite{Lion}]\label{lion}
	Let $\mu$ and $\nu$ are two positive bounded measures on $\mathbb{R}^N$ satisfying $$\left(\int_{\mathbb{R}^{N}}|\phi|^h d\nu\right)^{\frac{1}{h}}\leq C\left(\int_{\mathbb{R}^{N}}|\phi|^t d\mu\right)^{\frac{1}{t}},~~\forall\phi\in C_c^\infty(\mathbb{R}^N),$$ for $1\leq t<h\leq\infty$ and for some $C>0$. Then, there exist a countable set $I$, a collection of distinct points $\{x_i:i\in I\}\subset \mathbb{R}^N$ and $\{\nu_i:i\in I\}\subset(0,\infty)$ such that
	$$\nu=\sum_{i\in I}\nu_i\delta_{x_i},~~\mu\geq C^{-t}\sum_{i\in I}\nu_i^{\frac{t}{h}}\delta_{x_i}.$$
\end{lemma}
\begin{lemma}[Br\'{e}zis-Lieb lemma, \cite{Brezis}]\label{brezis}
	Let $u_n\rightarrow u$ a.e. and $u_n\rightarrow u$ weakly in $L^p(\Omega)$  for all $n$ where $\Omega\subset\mathbb{R}^N$ and $0<p<\infty$. Then,  
	$$\lim_{n\rightarrow \infty}\left(\int_{\Omega}|u_n|^p-\int_{\Omega}|u_n-u|^p\right)=\int_{\Omega}|u|^p.$$	
\end{lemma}
\begin{lemma}\label{bah}
	Consider the mapping $I_1:W_0\rightarrow W_0^*$ defined as
	$$\langle I_1(u),v\rangle=\int_{\mathbb{R}^{2N}}\frac{|u(x)-u(y)|^{p(x,y)-2}(u(x)-u(y))(v(x)-v(y))}{|x-y|^{N+sp(x,y)}}dxdy,$$ for every $u,v\in W_0$. Then, the following properties hold for $I_1$.
	\begin{enumerate}
		\item $I_1$ is a bounded and strictly monotone operator.
		\item $I_1$ is a mapping of $(S_+)$ type, i.e. if $u_n\rightarrow u$ weakly in $W_0$ and $\lim_{n\rightarrow \infty}\sup\langle I_1(u_n)-I_1(u),u_n-u\rangle\leq0$, then $u_n\rightarrow u$ strongly in $W_0$.
		\item $I_1:W_0\rightarrow W_0^*$ is a homeomorphism.
	\end{enumerate}
\end{lemma}
\begin{remark}
	The Lemma $\ref{bah}$ is a generalization of the Lemma 4.2 in \cite{Bahrouni2} where the authors worked with the space $\{u\in W^{s,p(\cdot,\cdot)}(\Omega):u=0$ in $\partial\Omega\}$. The proof here follows exactly the same arguments even in this case.
\end{remark}
\section*{Conclusion}
We have developed a concentration compactness type principle (CCTP) in a variable exponent setup. The symmetric mountain pass lemma and the CCTP have been applied to a problem involving fractional $(p(x),p^+)$-Laplacian to guarantee the existence of infinitely many nontrivial solutions.
\section*{Data availability statement}
The manuscript uses no data which needs to be shared or reproduced.
\section*{Acknowledgement}
The author Akasmika Panda thanks the financial assistantship received from the Ministry of Human
Resource Development (M.H.R.D.), Govt. of India. Both the authors also acknowledge the facilities received from the Department of mathematics, National Institute of Technology Rourkela. The authors thank Anouar Bahrouni for the useful discussions.
 
\end{document}